\documentclass{amsart}
\usepackage{amsfonts,amssymb,amsmath,amsgen,amsthm}
\usepackage{hyperref}
\usepackage{color}
\usepackage{mathtools}
\newcommand{\msc}[2][2000]{%
  \let\@oldtitle\@title%
  \gdef\@title{\@oldtitle\footnotetext{#1 \emph{Mathematics subject
        classification.} #2}}%
}

\theoremstyle{plain}
\newtheorem{theorem}{Theorem} [section]
\newtheorem{definition}[theorem]{Definition}

\newtheorem{lemma}[theorem]{Lemma}
\newtheorem{corollary}[theorem]{Corollary}
\newtheorem{proposition}[theorem]{Proposition}

\theoremstyle{remark}
\newtheorem{remark}[theorem]{Remark}

\def\C{{\mathbb C}}% complex numbers
\def\R{{\mathbb R}}% real numbers
\def\N{{\mathbb N}}% nonnegative integers
\def\Z{{\mathbb Z}}% integers
\def\O{\mathcal O}
\def\F{\mathcal F}

\def\Id{\mathbf 1}

\def\({\left(}
\def\){\right)}
\def\<{\left\langle}
\def\>{\right\rangle}
\def\le{\leqslant}
\def\ge{\geqslant}

\def\Eq#1#2{\mathop{\sim}\limits_{#1\rightarrow#2}}
\def\Tend#1#2{\mathop{\longrightarrow}\limits_{#1\rightarrow#2}}

\def\d{{\partial}}
\def\eps{\varepsilon}

\def\si{{\sigma}}

\allowdisplaybreaks
\numberwithin{equation}{section}

\begin{document}
\title[Global splitting for NLS]
{Scattering and uniform in time error estimates  for splitting method in NLS}

\author[R. Carles]{R\'emi Carles}
\author[C. Su]{Chunmei Su}

\address{Univ Rennes, CNRS\\ IRMAR - UMR 6625\\ F-35000
  Rennes, France}
\email{Remi.Carles@math.cnrs.fr}

\address{Yau Mathematical Sciences Center\\Tsinghua University\\10084 Beijing, China}
\email{sucm@tsinghua.edu.cn}

\begin{abstract}
  We consider the nonlinear Schr\"odinger equation with a defocusing
  nonlinearity which is mass-(super)critical and
  energy-subcritical. We prove uniform in time error estimates for the
 Lie-Trotter  time splitting discretization. This uniformity in time
 is obtained thanks to a vectorfield which provides time decay
 estimates for the exact and numerical solutions. This vectorfield is
 classical in scattering theory, and requires several technical
 modifications compared to previous error estimates for splitting
 methods.
\end{abstract}
\thanks{RC was supported by Rennes M\'etropole through its AIS
  program, and by Centre Henri Lebesgue, program ANR-11-LABX-0020-0.}

\maketitle

%%%%%%%%%%%%%%%%%%%%%%%%%%%%%%%%%%%%%%%%%%%%%%%%%%%%%%%%%%%%%%%%%%%%%%%%%%%%%%%%%%%%%%%%%%%%%

\section{Introduction}\label{sec:intro}

We consider large time error estimates for the Lie-Trotter time
splitting method associated to the defocusing nonlinear Schr\"odinger equation
    \begin{equation}\label{eq:NLS}
       i \partial_t u +\frac{1}{2} \Delta u =  |u|^{2\si} u, \quad
       (t,x) \in \R\times \R^d, \quad
    u_{\mid t=0}= \phi,
  \end{equation}
  in space dimension $d\le 5$, in the case where the nonlinearity is
  mass-(super)critical and
  energy-subcritical,
  \begin{equation}\label{eq:sigma}
     \frac{2}{d}\le\sigma<\frac{2}{(d-2)_+},
  \end{equation}
that is, $\si\ge 2/d$ when $d\le 2$, and $2/d\le \si<2/(d-2)$ when
$d\ge 3$.
The restriction on the space dimension is due to the fact that we want
the nonlinearity to be energy-subcritical, and to have  two
continuous derivatives, $\si>1/2$. Under these
assumptions, the Cauchy problem \eqref{eq:NLS} is globally well-posed
in $H^1(\R^d)$ (\cite{CazCourant, GV79Cauchy}), and
mass and energy are conserved,
\begin{align*}
  & \text{Mass: }M(u(t)):=\|u(t)\|_{L^2(\R^d)}^2 = M(\phi),\\
  &\text{Energy: } E(u(t)):=\frac{1}{2}\|\nabla
    u(t)\|_{L^2(\R^d)}^2+\frac{1}{\si+1}\|u(t)\|_{L^{2\si+2}(\R^d)}^{2\si+2} =E(\phi).
\end{align*}
Like initially in \cite{GV79Scatt}, denote by $\Sigma\subset
H^1(\R^d)$ the Hilbert space, sometimes called \emph{conformal space}
in the context of nonlinear Schr\"odinger equations,
\begin{equation*}
  \Sigma:=\left\{\phi\in H^1(\R^d);\ \int_{\R^d}|x|^2|\phi(x)|^2dx<\infty\right\},
\end{equation*}
equipped with the norm
\begin{equation*}
    \|\phi\|_{\Sigma}^2:=
    \|\phi\|_{L^2(\R^d)}^2 + \|\nabla \phi\|_{L^2(\R^d)}^2+ \|x
    \phi\|_{L^2(\R^d)}^2.
\end{equation*}
Then the Cauchy problem \eqref{eq:NLS} is globally well-posed in
$\Sigma$ as well: if $\phi\in \Sigma$, the solution $u(t,\cdot)$ has a finite momentum in $L^2(\R^d)$
for all time, and the evolution of this quantity is described by the
\emph{pseudo-conformal evolution law}, recalled in
Proposition~\ref{prop:pseudoconf}.
\smallbreak

We now recall the definition of the Lie-Trotter time splitting for
\eqref{eq:NLS}.
  We define $N(t) \phi$ as the solution of the flow
    \begin{equation*}
       i\partial_t u  = |u|^{2\si} u,\quad
    u_{\mid t=0}= \phi,
    \end{equation*}
that is, $N(t) \phi = \phi e^{-i t |\phi|^{2\si}}$. We set $S(t) \phi$ as the solution of the linear Schr\"odinger flow
\begin{equation*}
    i \partial_t u +\frac{1}{2} \Delta u =  0, \quad
    u_{\mid t=0}= \phi.
\end{equation*}
It is a Fourier multiplier, $S(t) \phi =
e^{i\frac{t}{2}\Delta} \phi$, and thus $S(t)$ is unitary on
$H^s(\R^d)$ for any $s\in \R$. The Lie-Trotter approximation is
defined, for $\tau\in (0,1)$,  as
\[Z(n \tau) \phi = \( S(\tau) N(\tau) \)^n \phi.\]
Error estimates for this time discretization were established first in
\cite{BBD} for globally Lipschitz nonlinearities. C.~Lubich
\cite{Lu08} proved error estimates in the case of the Strang
splitting, allowing (Schr\"odinger-Poisson nonlinearity and) cubic
nonlinearity ($\si=1$ in \eqref{eq:NLS}), hence a nonlinearity which
is not globally Lipschitz continuous.
\smallbreak

As pointed out  in \cite{Ignat2011,IgnatZuazua2006,IgnatZuazua2009,StKe05},  $S(\cdot)$ does not satisfy
discrete in time Strichartz estimates, which makes it difficult to
envisage error estimates involving a rather low regularity (in space)
of the initial datum $\phi$; see also \cite{FaouGrebert2011,ORS21} for
discussions leading to the same conclusion, that $S(\cdot)$ should be
modified in order to get better convergence results.
Following \cite{Ignat2011,IgnatZuazua2009} and the adaptation in
\cite{ChoiKoh2021}, we consider the modified splitting operator:
    \begin{equation}\label{eq:modified-splitting}
    Z_{\tau} (n\tau) = \( S_{\tau} (\tau) N(\tau) \)^n \Pi_{\tau}\phi.
    \end{equation}
Here, $S_{\tau}(t)$ denotes the frequency localized Schr\"odinger flow given by
 \[
    S_{\tau}(t) \phi = S(t) \Pi_{\tau} \phi,
 \]
where
    \begin{equation}\label{eq:cut-off}
    \widehat{\Pi_{\tau} \phi} (\xi)
    = \chi (\tau^{1/2}\xi) \widehat{\phi}(\xi) ,\quad \xi \in \R^d,
    \end{equation}
and $\chi \in C^k(\R^d)$ is a cut-off function supported in $B^d
(0,2)$ such that $\chi \equiv 1$ on $B^d (0,1)$, $k$ an integer larger
than $1+d/2$
(this condition appears in the proof of Lemma~\ref{lem:cutoff}).
\smallbreak

All the error estimates associated to splitting methods for nonlinear
Schr\"odinger equations have been established so far for bounded time
intervals $[0,T]$ (with constants growing at best exponentially in
$T$).
As far as we are aware of, the same is true
regarding linear Schr\"odinger equations with potential.
On the other hand,
global in time estimates have been proven in the framework of kinetic
equations \cite{FHR18}. There,
time decay estimates (associated to a  scattering phenomenon) play
a crucial role.
In this paper, we prove an analogous result in the
case of \eqref{eq:NLS}, by using techniques related to the scattering
theory (in $\Sigma$), in order to get quantitative time decay
estimates) for nonlinear Schr\"odinger equations. More precisely, we
use a specific vectorfield, standard in the scattering theory for
\eqref{eq:NLS}, $J(t)=x+it\nabla$, which provides more precise decay
estimates in time than the mixed $L^q_tL^r_x$-norms appearing in
Strichartz estimates. It is well known that $J$ does not commute with
$S$, but
$J(t) = S(t)xS(-t)$  (see Proposition~\ref{prop:pseudoconf} below). A
new technical specific aspect in this paper is
that we also have to deal with the absence of commutation between $J$
and the frequency cut-off $\Pi_\tau$.

For any interval $I \subset [0,\infty)$, we define the space $\ell^q
(n \tau \in I;\, L^r (\R^d))$, or simply $\ell^q
( I;\, L^r )$ as
consisting of functions defined on $\tau \mathbb{Z} \cap I$ with values in $L^r (\R^d)$, the norm of which is given by
\begin{equation}\label{eq:ellqLr}
  \|u\|_{\ell^q (I;\, L^r )} = \left\{
    \begin{aligned}
      \Big( \tau \sum_{n\tau \in I} \|u(n \tau)\|_{L^r (\R^d)}^q \Big)^{1/q}
      & \qquad\text{ if }1\le q<\infty,\\
      \sup_{n\tau\in I} \|u(n \tau)\|_{L^r (\R^d)}& \qquad\text{ if } q=\infty.
    \end{aligned}
    \right.
\end{equation}
We recall the notion of admissible pairs in the context of
Schr\"odinger equation (we shall not need endpoint Strichartz
estimates, $(q,r)=(2,\frac{2d}{d-2})$ for $d\ge 3$).

\begin{definition}\label{def:adm}
 A pair $(q,r)$ is admissible if $2\le r
  <\frac{2d}{d-2}$ ($2\le r\le\infty$ if $d=1$, $2\le r<
  \infty$ if $d=2$)
  and
\[\frac{2}{q}=\delta(r):= d\left( \frac{1}{2}-\frac{1}{r}\right).\]
\end{definition}
\begin{remark}
  We note that the range for $q$ is equivalent to: $q\in (2,\infty]$ if
$d\ge 2$, and $q\in [4,\infty]$ if $d=1$.
\end{remark}
As it plays a central role in the analysis of \eqref{eq:NLS},
throughout this paper, and following \cite{ChoiKoh2021,Ignat2011}, we
denote by $(q_0, r_0)$ the admissible pair
\begin{equation*}
  (q_0,r_0)=\(\frac{4\si+4}{d\si},2\si+2\).
\end{equation*}

\begin{theorem}\label{theo:CVL2}
Let $d\le 5$, $\si$ satisfying \eqref{eq:sigma}, with in addition
$\si\ge 1/2$ (if $d=5$), $\phi\in \Sigma$, and $u\in C(\R;\Sigma)$ the
solution to \eqref{eq:NLS}. Suppose there exists $M_2$ such that
  \begin{equation}\label{eq:aprioriZ}
    \max_{A\in \{\Id,\nabla,J\}}\| A(n\tau)Z_{\tau}(n \tau)\|_{\ell^{\infty} (\tau\N; L^2)} +
	\| Z_{\tau}(n \tau)\|_{\ell^{q_0} (\tau\N; W^{1,r_0} )}\le M_2,
      \end{equation}
      where $J(t)=x+it\nabla$. Then there exists
      $C=C(d,\si,\|\phi\|_{\Sigma},M_2)$ such that for all $\tau\in
      (0,1)$,
      \begin{equation*}
       \sup_{n\ge 0} \|Z_\tau(n\tau)-u(n\tau)\|_{L^2(\R^d)}\le
       C\tau^{1/2}.
     \end{equation*}
     In addition, there exists $u_+\in \Sigma$ such that
     \begin{equation*}
       \lim_{k\to \infty}\sup_{n\ge k}\|Z_\tau(n\tau) -
       S(n\tau)u_+\|_{L^2(\R^d)}\le C\tau^{1/2}.
     \end{equation*}
\end{theorem}

\begin{remark}
  We prove convergence in the $L^2$-topology. It is very likely that
  it holds also in the $H^1$-topology if we require in addition  $\phi\in H^2$, in view of the second result in Theorem \ref{theo:B}.
\end{remark}

\begin{remark}
  The existence of the asymptotic state $u_+$ can be understood as
  follows: for sufficiently large time, nonlinear effects have become
  negligible, and the action of the nonlinear flow $N(\tau)$ converges
  (fast enough) toward the identity, so the splitting operator $Z_\tau$
  behaves like $S_\tau$, which in turn is equivalent to $S$ for smooth
  functions. This is the meaning of the last estimate in
  Theorem~\ref{theo:CVL2} which, in some sense, makes the uniform error
  estimate more precise.
\end{remark}

Now demanding $\si>1/2$, and not only $\si\ge
  1/2$ (see the proof of Proposition~\ref{prop:discrete-stability} for the
    reason why this constraint is introduced),
we show that the
assumptions of Theorem~\ref{theo:CVL2} are indeed satisfied by the
numerical solution:
\begin{theorem}\label{theo:stability}
  Let $d\le 5$, $\si$ satisfying \eqref{eq:sigma},
with in addition
$\si> 1/2$ (if $d=4$ or $5$), and $\phi\in \Sigma$.
 Then, for any
admissible pair $(q,r)$,
there exists $C(d,\si,q,\phi)$ such that  for all $\tau\in (0,1)$, the
numerical solution $Z_\tau$ satisfies
\begin{equation*}
  \max_{A\in
    \{\Id,\nabla,J\}}\|A(n\tau)Z_\tau(n\tau)\|_{\ell^q(\tau\N;L^r)}\le
  C(d,\si,q,\phi).
\end{equation*}
\end{theorem}

\begin{remark}\label{rem:non-disp}
 In view of classical results on nonlinear Schr\"odinger equations
 (see e.g \cite{CazCourant,TaoDisp}), Theorems~\ref{theo:CVL2} and
 \ref{theo:stability} remain valid in the case of a focusing
 nonlinearity ($|u|^{2\si}u$ is replaced by $-|u|^{2\si}u$ in
 \eqref{eq:NLS}), provided that $\|\phi\|_{\Sigma}$ is sufficiently
 small. Without smallness assumption, finite time blow up is possible,
 but even global solutions need not be dispersive, since standing wave
 solutions of
 the form $u(t,x)=e^{i\omega t}\phi(x)$ exist. Theorems~\ref{theo:CVL2} and
 \ref{theo:stability}  highly rely on dispersive properties of the
 solution to \eqref{eq:NLS}, and should not be expected to remain true
 in the case of standing waves. More general nonlinearities than the
 homogeneous one considered in \eqref{eq:NLS} could  be addressed though
 (typically, combined power nonlinearities), provided that suitable a
 priori estimates (in the spirit of the pseudoconformal conservation
 law recalled in Proposition~\ref{prop:pseudoconf}) are
 available. This is for instance the case when the nonlinearity is the
 sum of two defocusing  homogeneous terms,
 $|u|^{2\si_1}u+|u|^{2\si_2}u$, $\frac{2}{d}\le
 \si_j<\frac{2}{(d-2)_+}$, but the situation is more involved when it
 is the sum of a focusing and a defocusing term, since standing waves
exist (see e.g. \cite{KOPV17}).
\end{remark}

\begin{remark}
As recalled in Section~\ref{sec:exact}, the assumptions on $\si$ and
$\phi$ ensure that the solution $u$ to \eqref{eq:NLS} is global, and
satisfies $u\in   L^q(\R;L^r(\R^d))$ for all admissible pairs. It is
tempting to conjecture that Theorem~\ref{theo:stability} remains
true under the assumptions that $u\in   L^q(\R;L^r(\R^d))$ for all
admissible pairs, a property that actually follows from the weaker one
$u\in   L^{q_0}(\R;L^{r_0}(\R^d))$, see
Section~\ref{sec:exact}. However, the introduction of the operator $J$
induces stronger estimates, and it is not clear at all that the proof
of Theorem~\ref{theo:stability}  can be adapted to this broader
setting (typically, one could assume $\phi\in H^1(\R^d)$ only).
 \end{remark}
\subsection*{Organization of the paper}

The rest of this paper is organized as follows. In Section~\ref{sec:exact}, we recall some
standard results related to \eqref{eq:NLS}, including global in time
estimates, establish some new ones (in particular
Theorem~\ref{theo:B}), to provide general estimates on
solutions to \eqref{eq:NLS}. Section~\ref{sec:numerical} is devoted to
general estimates involving the numerical solution (discrete in
time). The main novelties concern the introduction of the operator
$J$, and the difficulties caused by its absence of commutation with the
frequency cut-off $\Pi_\tau$. Theorem~\ref{theo:CVL2} is proved in
Section~\ref{sec:CVL2}. In Section~\ref{sec:local}, we establish local
$\Sigma$ stability results (Proposition~\ref{prop:local-stab}); there,
local means local in time, allowing a neighborhood to $t=\infty$. In
Section~\ref{sec:more}, we prove refined estimates allowing to
conclude with the proof of Theorem~\ref{theo:stability} in
Section~\ref{sec:stability}.
Finally Section~\ref{sec:conclusion} is dedicated to a summary and a list of
possible related extensions.

\subsection*{General notations}
\begin{itemize}
\item We denote by $\Id$  the identity operator.
\item We denote by $L^r$ the standard space $L^r(\mathbb{R}^d)$ for $1\le r\le \infty$.
\item For $y\in \R^n$, $n\in \N$, the Japanese bracket is
  $\<y\>=(1+|y|^2)^{1/2}$.
\item We denote by $C$ generic constants, which may vary from line to
  line.
\item We underline the dependence of constants as follows: $C(d,\si)$ or $C_{d, \si}$
  means that $C$ does not depend on other parameters, such as $\tau$
  or $\phi$.
\item For $a,b\ge0$, we use the notation
 \begin{equation*}
    a\lesssim b
 \end{equation*}
whenever there exists a constant $C$ independent of $\tau\in (0,1)$ and the
time interval considered (but certainly depending on $d$ and $\si$)
such that $a\le Cb$.
\item When considered useful, we write $u^\phi$ to underscore that $u$
  is the solution to \eqref{eq:NLS} with initial datum $\phi$
  (typically, when several such solutions are involved).
\end{itemize}

%%%%%%%%%%%%%%%%%%%%%%%%%%%%%%%%%%%%%%%%%%%%%%%%%%%%%%%%%%%%%%%%%%%%%%%%%%%%%%%%%%%%%%%%%%%%%

\section{Preliminary estimates: the exact solution}
\label{sec:exact}

\subsection{Generalities}
For $t_0\in\R$, Duhamel formula for the solution $u$ to
\eqref{eq:NLS} with the initial condition $u_{\mid
  t=t_0}=\phi$, reads as follows:
 \begin{equation}\label{eq:duhamel}
    u(t) = S(t-t_0) \phi - i \int_{t_0}^t S(t-s) \(|u|^{2\si} u\) (s) ds.
\end{equation}
The standard Strichartz inequalities associated to the
Schr\"odinger equation (see e.g. \cite{CazCourant,TaoDisp}) are
summarized below. We recall
that the notion of admissible pairs was introduced in Definition~\ref{def:adm}.
\begin{proposition}[Strichartz estimates]\label{prop:strichartz}
  Let $d\ge 1$ and $S(t)=e^{i\frac{t}{2}\Delta}$. \\
$(1)$ \emph{Homogeneous estimates.} For any admissible pair $(q,r)$, there exists $C_{q}$  such that
\begin{equation*}
 % \label{eq:stri-homo}
  \|S(t)\phi\|_{L^q(\R;L^r)} \le C_q
\|\phi \|_{L^2},\quad \forall \phi\in L^2.
\end{equation*}
$(2)$ \emph{Inhomogeneous estimates.}
Denote
\begin{equation*}
  D(F)(t,x) = \int_0^t S(t-s)F(s,x)ds.
\end{equation*}
For all admissible pairs $(q_1,r_1)$ and~$
    (q_2,r_2)$,  there exists $C=C_{q_1,q_2}$ such that for all
    intervals $I\ni 0$,
\begin{equation}\label{eq:strichnl}
      \left\lVert D(F)
      \right\rVert_{L^{q_1}(I;L^{r_1})}\le C \left\lVert
      F\right\rVert_{L^{q'_2}\(I;L^{r'_2}\)},\quad \forall F\in L^{q'_2}(I;L^{r'_2}).
\end{equation}
\end{proposition}
With Strichartz and H\"older inequalities in mind, we remark that
$r_0$ satisfies
\begin{equation*}
  \frac{1}{r_0'}=\frac{2\si+1}{r_0},
\end{equation*}
with $2\si+1$ being the homogeneity of the nonlinearity in \eqref{eq:NLS}, and we introduce $\gamma$ given by
\begin{equation}\label{eq:gamma}
  \frac{1}{q_0'}=\frac{1}{q_0}+\frac{2\si}{\gamma}\Longleftrightarrow  \gamma=
  \frac{4\si(\si+1)}{2-(d-2)\si}.
\end{equation}
We see that $\gamma$ is finite since the nonlinearity is
energy-subcritical. The above relations will be applied many times, as
follows (according to whether continuous or discrete time intervals
are considered), recalling that $2\si\ge 1$:
\begin{equation}
  \label{eq:holder}
  \begin{aligned}
      &\left\| |f|^{2\si-1}gh\right\|_{L^{q_0'}(I;L^{r_0'})}\le
  \|f\|_{L^\gamma(I;L^{r_0})}^{2\si-1}
  \|g\|_{L^{\gamma}(I;L^{r_0})}\|h\|_{L^{q_0}(I;L^{r_0})},\quad
  \text{or}\\
  & \left\| |f|^{2\si-1}gh\right\|_{\ell^{q_0'}(I;L^{r_0'})}\le
  \|f\|_{\ell^\gamma(I;L^{r_0})}^{2\si-1}
  \|g\|_{\ell^{\gamma}(I;L^{r_0})}\|h\|_{\ell^{q_0}(I;L^{r_0})}.
\end{aligned}
\end{equation}
The following result was discovered in \cite{GV79Scatt}, and is
crucial to turn the local error estimates from
\cite{Ignat2011,ChoiKoh2021} into global ones:

\begin{proposition}[Pseudo-conformal conservation law]\label{prop:pseudoconf}
The operator
\[J(t)=x+it\nabla\]
satisfies the following properties:
  \begin{itemize}
  \item $J(t) =S(t)xS(-t)$, and therefore $J$ commutes with the
 linear part of \eqref{eq:NLS},
\begin{equation}\label{eq:Jcommute}
\left[ J(t),i\d_t +\frac{1}{2}\Delta\right]=0\, .
\end{equation}
\item It can be factorized as
 \begin{equation*}%\label{eq:Jfactor}
J(t)= i t \, e^{i\frac{|x|^2}{2t}}\nabla\Big( e^{-i\frac{|x|^2}{2t}}\,
\cdot\Big)\, .
\end{equation*}
As a consequence, $J$ yields weighted Gagliardo-Nirenberg
inequalities. For
$2\le r <\frac{2d}{(d-2)_+}$ ($2\le r\le \infty$ if $d=1$), there exists
$C(d, r)$ depending only on $d$ and $r$ such that
\begin{equation}\label{eq:GNlibre}
\left\| f \right\|_{L^r}\le \frac{C(d, r)}{|t|^{\delta(r)}} \left\|
f \right\|_{L^2}^{1-\delta(r)} \left\|
J(t) f \right\|_{L^2}^{\delta(r)},\quad \delta(r):=d\(
\frac{1}{2}-\frac{1}{r}\)\, .
\end{equation}
Also, if $F(z)=G(|z|^2)z$ is $C^1$, then $J(t)$
acts like a derivative on $F(w)$:
\begin{equation}\label{eq:Jder}
J(t)\(F(w)\) = \d_z F(w)J(t)w -\d_{\overline z} F(w)\overline{ J(t)w
}\, .
\end{equation}
\end{itemize}
Any solution $u\in C(\R;\Sigma)$ to \eqref{eq:NLS} satisfies the
pseudo-conformal conservation law:
\begin{equation}\label{eq:pseudoconf}
\frac{d}{dt}\left(\frac{1}{2}\| J(t)u
\|^2_{L^2}+\frac{t^2}{\si+1}\|u
    (t)\|^{2\si+2}_{L^{2\si +2}}\right)=\frac{t}{\si+1}(2-d\si)\|u
    (t)\|^{2\si+2}_{L^{2\si+2}}.
\end{equation}
\end{proposition}
For some time interval $I$,  we introduce the norm
\begin{equation}\label{eq:Xnorm}
 \|u\|_{X(I)} =\max_{A\in \{\Id,\nabla,J\}}\sup_{t\in I}\|A(t)u(t)\|_{L^2},
\end{equation}
and the space $X(I)$ defined by the finiteness of this quantity. We
emphasize that in view of the time dependence of $J$,
$X(I)=L^\infty(I;\Sigma)$ \emph{if and only if $I$ is bounded}. This
can be seen on the linear Schr\"odinger
equation (linear solutions $S(t)\phi$ are dispersive): in view of
\eqref{eq:Jcommute}, since $S(t)$ is unitary on $L^2$,
\begin{equation*}
  \|J(t)S(t)\phi\|_{L^2}= \|x\phi\|_{L^2},\quad \|\nabla
  S(t)\phi\|_{L^2}=\|\nabla \phi\|_{L^2}\Longrightarrow \|x
  S(t)\phi\|_{L^2}\Eq t \infty t\|\nabla\phi\|_{L^2}.
\end{equation*}
 We
note that if $I$ is reduced to a single element, $I=\{t_0\}$, then
$X(I)=\Sigma$, but the norm $\|\cdot \|_{X(I)}$ involves the operator
$J(t_0)$, so the norms $\|\cdot\|_\Sigma$ and $\|\cdot \|_{X(I)}$ are
  equivalent, but with comparing constants highly depending on $t_0$. We
also consider the norm
\begin{equation}
  \label{eq:Ynorm}
  \|u\|_{Y(I)}:= \max_{A\in \{\Id,\nabla,J\}}\sup_{(q,r)\text{
      admissible}}\|A u\|_{L^q(I;L^{r})},
\end{equation}
and the associated space $Y(I)$. Note
that $\|u\|_{X(I)}\le \|u\|_{Y(I)}$.

\subsection{Global solutions}

At this stage, we emphasize that the assumption $\si\ge 2/d$ implies
that for $t\ge 0$, the right hand side in \eqref{eq:pseudoconf} is
nonpositive. Together with the conservation of mass and energy, this
entails $u\in X(\R)$. This implies in
particular $u\in L^{\gamma}(\R;L^{r_0})$, as
\begin{equation}
  \label{eq:gammadeltar0}
\gamma\delta(r_0)>1.
\end{equation}
This property is indeed classical in scattering theory for
\eqref{eq:NLS}: it is equivalent to
\begin{equation*}
  \si>\frac{2-d+\sqrt{d^2+12d+4}}{4d}=:\si_*,
\end{equation*}
and $\si_*<2/d$. This parameter $\si_*$ is a standard lower bound for
$\sigma$ to prove scattering theory in $\Sigma$, see
e.g. \cite{CazCourant} (see also \cite{CW92,NakanishiOzawa} for the case $\si=\si_*$).
The standard Gagliardo-Nirenberg inequality (for
bounded $t$) and \eqref{eq:GNlibre} (for large $t$) yield
\begin{equation*}
  \|u(t)\|_{ L^{r_0}}\lesssim\<t\>^{-\delta(r_0)}
  \|u(t)\|_{L^2}^{1-\delta(r_0)} \(\|\nabla u(t)\|_{L^2}^{\delta(r_0)}
  +\|J(t)u(t)\|_{L^2}^{\delta(r_0)} \) ,
\end{equation*}
and  in view of
H\"older inequality in time,
\begin{equation}
  \label{eq:apriori-univ}
  \|u\|_{ L^{\gamma}(I;L^{r_0})}\le C\left\|\<t\>^{-\delta(r_0)}\right\|_{L^\gamma(I)}
  \|u\|_{X(I)},
\end{equation}
where $C$ does not depend on the time interval $I$.
For $A\in \{\Id,\nabla,J\}$,  Strichartz and
H\"older inequalities
(involving \eqref{eq:holder}) then yield
\begin{equation*}
 \sup_{(q,r)\text{ admissible}}\|Au\|_{L^q(I;L^{r})} \lesssim
 \|\phi\|_{\Sigma} + \|u\|_{L^\gamma(I;L^{r_0})}^{2\si} \sup_{(q,r)\text{
     admissible}}\|Au\|_{L^q(I;L^{r})} ,
\end{equation*}
and splitting $\R$ into finitely many intervals $I_j$ where
$\|u\|_{L^\gamma(I_j;L^{r_0})}$ is sufficiently small, we infer $u\in
Y(\R)$. We obtain the  following statement (see
e.g. \cite[Theorem~7.4.1]{CazCourant} or  \cite[Theorem~B]{HT87}):
\begin{theorem}\label{theo:aprioriX}
  Let $\phi\in \Sigma$, $\frac{2}{d}\le \si<\frac{2}{(d-2)_+}$. Then \eqref{eq:NLS}
  has a unique solution $u\in C(\R;\Sigma)\cap L^{q_0}_{\rm
    loc}(\R;L^{r_0})$. It satisfies $u\in Y(\R)$,
  and there exist $u_\pm\in
  \Sigma$ such that
  \begin{equation*}
    \left\|S(-t)u(t)-u_\pm\right\|_{\Sigma}\Tend t {\pm \infty} 0.
  \end{equation*}
\end{theorem}
A priori, the asymptotic states $u_+$ and $u_-$ are different, even
though the relation between $u_+$ and $u_-$ remains rather
mysterious in general (see e.g. \cite{CaDPDE}).
\begin{remark}\label{rem:restriction}
The first reason why we assume $\si\ge 2/d$ instead of the more general
hypothesis $\si>\sigma_*$ ensuring scattering in $\Sigma$ is that in
view of \eqref{eq:pseudoconf}, we have the global estimate $Ju\in
L^\infty(\R;L^2(\R^d))$ as soon as $\si\ge 2/d$. On the other hand, if
$\si_*<\si<2/d$,   \eqref{eq:pseudoconf} provides only a control on
the growth (in $t$) of $t^2\|u(t)\|_{L^{2\si+2}}^{2\si+2}$ (via
Gronwall lemma), hence of $\|J(t)u\|_{L^2}$ (using
\eqref{eq:pseudoconf} again). The assumption $\si\ge 2/d$ is made
not only for this simplification: filling the gap
$\si_*<\si<2/d$ would require to modify several arguments below,
see the proofs of Theorem~\ref{theo:B} and Lemma~\ref{lem:source}.
\end{remark}
We state and prove the global in time analogue of
\cite[Theorem~B]{ChoiKoh2021}, and establish some specific global in time
integrability properties:
\begin{theorem}\label{theo:B}
  Let $1\le d\le 5$, $\frac{2}{d}\le \si<\frac{2}{(d-2)_+}$, with
  in addition  (if $d=5$) $\si\ge 1/2$.
  \begin{itemize}
  \item For any $M\ge 1$, there exists $C=C(M,d,\si)$ such that for
    any $t_0\in \R$, if
    $\phi_1,\phi_2\in \Sigma$ are initial data for $u_1$ and
    $u_2$, respectively, at time $t_0$,
    \begin{equation*}
      u_{j\mid t=t_0}=\phi_j,\quad j=1,2,
    \end{equation*}
 and are   such that
    $\|\phi_1\|_{X(\{t_0\})},\|\phi_2\|_{X(\{t_0\})}\le M$, then
    \begin{equation*}
      \|u_1-u_2\|_{Y(\R)}\le C\|\phi_1-\phi_2\|_{X(\{t_0\})}.
    \end{equation*}
 \item If $\psi\in \Sigma\cap H^2$, and $u$   solves
   \eqref{eq:NLS}, where, for $t_0\in \R$,
   $u_{\mid t=t_0}=\psi$, then for all $A\in \{\Id,\nabla,J\}$,
    \begin{equation*}
     A u^\psi \in \bigcap_{(q,r)\text{
          admissible}}L^q(\R;W^{1,r}).
    \end{equation*}
  \end{itemize}
\end{theorem}
\begin{proof}
  For the first point,  we note that the conservation of the energy
  and the pseudoconformal conservation law \eqref{eq:pseudoconf} yield
  \begin{equation*}
    \|u_j\|_{X(\R)}\le C_1(M,d,\si),\quad j=1,2.
  \end{equation*}
  In view of (weighted) Gagliardo-Nirenberg inequality,
  \begin{equation*}
    \|u_j(t)\|_{L^{r_0}}\le \frac{C_2(M,d,\si)}{\<t\>^{\delta(r_0)}}.
  \end{equation*}
We then remark that since $J(t)=S(t)xS(-t)$ (Proposition~\ref{prop:pseudoconf}),
we have $J(t)=S(t-t_0)J(t_0)S(t_0-t)$, and Duhamel's formula becomes,
  for $A\in \{\Id,\nabla,J\}$, and $j=1,2$,
  \begin{equation*}
    A(t)u_j(t) = S(t-t_0)A(t_0)\phi_j-i\int_{t_0}^tS(t-s)A(s)\(|u_j|^{2\si}u_j\)(s)ds.
  \end{equation*}
  Considering the difference between the equations for $j=1$ and
  $j=2$, respectively, Strichartz and H\"older inequalities (like
  mentioned above) yield, if $t_0\in I$, then for all admissible $(q, r)$,
  \begin{align*}
   \|u_1-u_2\|_{L^{q}(I;L^{r})}&\lesssim
     \|\phi_1-\phi_2\|_{X(\{t_0\})} \\
    &\quad+
    \(\|u_1\|_{L^\gamma(I;L^{r_0})}^{2\si} +
    \|u_2\|_{L^\gamma(I;L^{r_0})}^{2\si}\)\|u_1-u_2\|_{L^{q_0}(I;L^{r_0})} \\
       &\lesssim
      \|\phi_1-\phi_2\|_{X(\{t_0\})} \\
      &\quad+ C_3(M,d,\si)\|\<t\>^{-\delta(r_0)}\|_{L^\gamma(I)}^{2\si}
 \|u_1-u_2\|_{L^{q_0}(I;L^{r_0})}.
  \end{align*}
When $A=\nabla$ or $J(t)$, computations are similar:
\begin{align*}
  \nabla \(|u_1|^{2\si}u_1\)- \nabla\(|u_2|^{2\si}u_2\)&= (\si+1)\( |u_1|^{2\si}\nabla u_1
  - |u_2|^{2\si}\nabla u_2\) \\
&\quad+\si\( |u_1|^{2\si-2}u_1^2\overline{\nabla u_1} -
  |u_2|^{2\si-2}u_2^2\overline{\nabla u_2}\);\\
 J(t) \(|u_1|^{2\si}u_1\)- J(t)\(|u_2|^{2\si}u_2\)&= (\si+1)\( |u_1|^{2\si} J(t) u_1
  - |u_2|^{2\si}J(t) u_2\) \\
&\quad-\si\( |u_1|^{2\si-2}u_1^2\overline{J(t) u_1} -
  |u_2|^{2\si-2}u_2^2\overline{J(t) u_2}\).
\end{align*}
We consider the first term of the right hand side, as the second term
is estimated in the same way:
\begin{align*}
  |u_1|^{2\si}Au_1  - |u_2|^{2\si}Au_2= |u_1|^{2\si}A(u_1-u_2) + \(
 |u_1|^{2\si}-|u_2|^{2\si}\)Au_2.
\end{align*}
Since $2\si\ge 1$, we have
\begin{equation*}
  \left| |u_1|^{2\si}-|u_2|^{2\si}\right|\lesssim
  \(|u_1|^{2\si-1}+|u_2|^{2\si-1}\)|u_1-u_2|.
\end{equation*}
Now we use \eqref{eq:holder} to find:
\begin{align*}
   \|A(u_1-&u_2)\|_{L^{q}(I;L^{r})}\lesssim
     \|\phi_1-\phi_2\|_{X(\{t_0\})} +\left\|A\(|u_1|^{2\si}u_1\)-
  A\(|u_2|^{2\si}u_2\)\right\|_{L^{q_0'}(I;L^{r_0'})} \\
  &\lesssim
     \|\phi_1-\phi_2\|_{X(\{t_0\})} \\
&\quad+
  \(\|u_1\|_{L^\gamma(I;L^{r_0})}^{2\si} +
    \|u_2\|_{L^\gamma(I;L^{r_0})}^{2\si}\)\|A(u_1-u_2)\|_{L^{q_0}(I;L^{r_0})}
  \\
&\quad+ \(\|u_1\|_{L^\gamma(I;L^{r_0})}^{2\si-1} +
    \|u_2\|_{L^\gamma(I;L^{r_0})}^{2\si-1}\)\|Au_2\|_{L^{q_0}(I;L^{r_0})}\|u_1-u_2\|_{L^\gamma(I;L^{r_0})}.
  \end{align*}
The last term is controlled, in view of \eqref{eq:apriori-univ}, by
\begin{equation*}
  \|u_1-u_2\|_{L^\gamma(I;L^{r_0})}\lesssim
  \left\|\<t\>^{-\delta(r_0)}\right\|_{L^\gamma(I)} \max_{B\in
    \{\Id,\nabla,J\}}\|B(u_1-u_2)\|_{L^\infty(I;L^2)}.
\end{equation*}
In view of Theorem~\ref{theo:aprioriX}, $Au_2\in L^{q_0}(\R;L^{r_0})$,
and we have:
\begin{align*}
  \max_{{A\in \{\Id,\nabla,J\}}\atop{(q,r)\,\,\,
  \text{admissible}}}&\|A(u_1-u_2)\|_{L^{q}(I;L^{r})}\lesssim
  \|\phi_1-\phi_2\|_{X(\{t_0\})}  \\
+  C_3(M,d,\si)&\|\<t\>^{-\delta(r_0)}\|_{L^\gamma(I)}^{2\si}
\max_{{A\in \{\Id,\nabla,J\}}\atop{(q,r)\,\,\,
  \text{admissible}}}\|A(u_1-u_2)\|_{L^{q}(I;L^{r})},
\end{align*}
and similarly, using Strichartz estimates again, for any $t_j\in \R$,
$I_j=\overline{[t_j,t_{j+1})} $
($t_j<t_{j+1}\le \infty$, we consider the adherence of $I_j$
to address a closed interval except if $t_{j+1}=\infty$),
\begin{align*}
  \max_{{A\in \{\Id,\nabla,J\}}\atop{(q,r)\,\,\,
  \text{admissible}}}&\|A(u_1-u_2)\|_{L^{q}(I;L^{r})}\le
  C_4\|u_1-u_2\|_{X(\{t_j\})}  \\
+  C_5(M,d,\si)&\left\|\<t\>^{-\delta(r_0)}\right\|_{L^\gamma(I)}^{2\si}
\max_{{A\in \{\Id,\nabla,J\}}\atop{(q,r)\,\,\,
  \text{admissible}}}\|A(u_1-u_2)\|_{L^{q}(I;L^{r})},
\end{align*}
where $C_{4}$ and $C_5$ is independent of $I_j$, that is,
\begin{align*}
  \|u_1-u_2\|_{Y(I)}\le
  C_4\|u_1-u_2\|_{X(\{t_j\})}
+  C_5(M,d,\si)\left\|\<t\>^{-\delta(r_0)}\right\|_{L^\gamma(I)}^{2\si}
\|u_1-u_2\|_{Y(I)}.
\end{align*}
Since
$\gamma\delta(r_0)>1$, we can split $[t_0,\infty)$ into
finitely many intervals $I_j$ such that
\begin{equation*}
   C_5(M,d,\si)\left\|\<t\>^{-\delta(r_0)}\right\|_{L^\gamma(I_j)}^{2\si}\le
   \frac{1}{2},\quad [t_0,\infty)=\bigcup_{j=0}^K I_j,
\end{equation*}
and then, for $j\ge 1$,
\begin{align*}
 \|u_1-u_2\|_{Y(I_j)} &\le C_4
\|u_1-u_2\|_{X(\{t_j\})}+\frac{1}{2}
   \|u_1-u_2\|_{Y(I_j)}\\
&\le 2  C_4\|u_1-u_2\|_{X(\{t_j\})}\le 2C_4 \|u_1-u_2\|_{Y(I_{j-1})} \\
&\le (2C_4)^{j+1}\|\phi_1-\phi_2\|_{X(\{t_0\})},
\end{align*}
by induction. Since $j\le K$, we infer
\begin{equation*}
  \|u_1-u_2\|_{Y([t_0,\infty))}\le C(M,d,\si)\|\phi_1-\phi_2\|_{X(\{t_0\})},
\end{equation*}
and the same obviously holds on $(-\infty,t_0]$. Hence the first point
of the theorem is established.
\smallbreak

  In order to prove the second point, we do not follow the same
  strategy as in \cite{CazCourant}, where the $H^2$ regularity of $u$
  is read from the equation \eqref{eq:NLS}, after it is proved that
  $\d_t u$ is in $L^2$ (this strategy does not require the
  nonlinearity to be more than $C^1$, and $\si>0$ is allowed).
  Since $\si\ge 1/2$, we may differentiate \eqref{eq:NLS} twice in
  space, instead of once in time.

  We check that since $d\si\ge 2$, $\gamma$, defined in
  \eqref{eq:gamma}, is such that $\gamma\ge q_0$, as
  \begin{equation*}
    \frac{1}{q_0}-\frac{1}{\gamma} = \frac{d\si-2}{4\si}.
  \end{equation*}
  Therefore, there exists $\rho\ge 2$ such that $(\gamma,\rho)$ is
  admissible, and
  \begin{equation*}
    d\(\frac{1}{\rho}- \frac{1}{r_0}\) =
    d\(\frac{1}{\rho}-\frac{1}{2}+\frac{1}{2}- \frac{1}{r_0}\)=
    \frac{2}{q_0}-\frac{2}{\gamma} = \frac{d\si-2}{2\si}=:s.
  \end{equation*}
  We infer that $W^{s,\rho}\hookrightarrow L^{r_0}$, with
  $s$ defined above, which is such that $s\in [0,1[$ since $2/d\le
  \si<2/(d-2)_+$. We note that for $A=J$, $\d_j$ and $A$ do not
  commute (but the commutation bracket is of order zero, and is therefore
  harmless in the estimates). Applying $A\in \{\Id,\nabla,J\}$ to
  the Duhamel's formula \eqref{eq:duhamel}, then $\d_j$, Strichartz
  estimates yield, for $t_0\in I$,
  \begin{equation}\label{eq:order2}
    \|\d_j A u\|_{L^{q}(I;L^{r})} \lesssim \|\d_j
    A(t_0)u(t_0)\|_{L^2} + \left\|\d_jA \(|u|^{2\si}u\)\right\|_{L^{q_0'}(I;L^{r_0'})}.
  \end{equation}
Now H\"older inequality entails
  \begin{align*}
    \left\|\d_jA \(|u|^{2\si}u\)\right\|_{L^{q_0'}(I;L^{r_0'})}& \lesssim
     \|u\|_{L^\gamma(I;L^{r_0})}^{2\si} \|\d_j A u\|_{L^{q_0}(I;L^{r_0})} \\
      &\quad +
        \|u\|_{L^\gamma(I;L^{r_0})}^{2\si-1}\|\d_j
        u\|_{L^\gamma(I;L^{r_0})}\|A u\|_{L^{q_0}(I;L^{r_0})}  \\
    &\lesssim \|u\|_{L^\gamma(I;L^{r_0})}^{2\si}
    \|\d_j A u\|_{L^{q_0}(I;L^{r_0})}   \\
    &\quad +
        \|u\|_{L^\gamma(I;L^{r_0})}^{2\si-1}\|\nabla
      u\|_{L^\gamma(I;W^{s,\rho})}\|Au\|_{L^{q_0}(I;L^{r_0})}   \\
       &\lesssim \|u\|_{L^\gamma(I;L^{r_0})}^{2\si}
      \|\d_j A u\|_{L^{q_0}(I;L^{r_0})}   \\
    &\quad +
        \|u\|_{L^\gamma(I;L^{r_0})}^{2\si-1}\|
      u\|_{L^\gamma(I;W^{2,\rho})}\|Au\|_{L^{q_0}(I;L^{r_0})}  .
  \end{align*}
If $\si>1/2$, we use the same idea as for the first point of the
theorem, namely
\begin{equation*}
\|u(t)\|_{L^\gamma(I; L^{r_0})}\lesssim \|\<t\>^{-\delta(r_0)}\|_{L^\gamma(I)},
\end{equation*}
and split $\R$ into finitely many intervals such that the last term
in \eqref{eq:order2} can be absorbed by the left hand side on each of
them. We conclude by (finite) induction like before.
If $\si=1/2$, we use in addition the property $Au\in
L^{q_0}(\R;L^{r_0})$ (from Theorem~\ref{theo:aprioriX}), and conclude
similarly, by considering Strichartz estimates involving other
admissible pairs, among them $(\gamma,\rho)$.
 \end{proof}

\subsection{Estimating the source term}

Suppose that $u$ solves \eqref{eq:NLS}, and $v$ solves
\begin{equation*}
 i\d_t v+\frac{1}{2}\Delta v = i\frac{N(\tau)-\Id}{\tau}v,\quad v_{\mid t=0}=\phi,
\end{equation*}
then Taylor formula shows that $w=u-v$ solves
\begin{equation*}
  i\d_t w+\frac{1}{2}\Delta w = i\frac{N(\tau)-\Id}{\tau}u-i
  \frac{N(\tau)-\Id}{\tau}v + r_w,
\end{equation*}
where the source term is controlled pointwise by $|r_w|\lesssim
|u|^{4\si+1}$. In terms of estimates, it is as if the power $2\si$ in
\eqref{eq:NLS} had been replaced by $4\si$, which may no longer
correspond to an energy-subcritical nonlinearity. This difficulty is
overcome thanks to the introduction of the frequency cutoff
$\Pi_\tau$.
The following lemma is a global in time version of
\cite[Lemma~2.8]{ChoiKoh2021}, and the proof resumes many of its ingredients:

\begin{lemma}\label{lem:source}
There exists $C$ independent of $\tau$ and the
time interval $I$
such that
    \begin{equation*}%\label{Sov_2p+1}
    \left\| |\Pi_{\tau} u|^{4\si+1} \right\|_{L^{q_0'} (I; L^{r_0'})}
    + \left\| |\Pi_{\tau} u|^{2\si} \Pi_{\tau} (|u|^{2\si} u) \right\|_{L^{q_0'} (I; L^{r_0'})}
    \le C \tau^{-1/2} \|u\|_{Y(I)}^{4\si+1},
  \end{equation*}
  where $Y(I)$ is defined by \eqref{eq:Ynorm}.
\end{lemma}
\begin{proof}
By Lemma~\ref{lem:bernstein} and H\"older inequality, we have
    \begin{equation}\label{eq:source-step1}
    \left\| |\Pi_{\tau} u|^{4\si+1} \right\|_{L^{q_0'} (I; L^{r_0'})}
    \le C_{d,\si} \tau^{-\frac{d}{2} (\frac{4\si+1}{r_1}-\frac{1}{r_0'})} \| u \|_{L^{(4\si+1)q_0'} (I; L^{r_1})}^{4\si+1}
    \end{equation}
for all $r_1 \le (4\si+1)r_0'$.
The value of $r_1 >0$ is chosen as
\[
    \frac{1}{r_1} =\frac{1}{4\si+1} \( \frac{2\si+1}{2\si+2}
    +\frac{1}{d} \) >\frac{1}{4\si+1} \frac{2\si+1}{2\si+2}=\frac{1}{(4\si+1)r_0'} .
  \]
   We check that in view of \eqref{eq:sigma}, $r_1\ge 2$. By definition,
   $r_1\ge 2$ if and only
  if
  \[
    d\si(4\si+3)\ge 2\si+2.
   \]
But $d\si\ge 2$, so the above inequality is satisfied. For $d=1,2$, we
do not have to check anything else. Introduce
 now $q_2$ and $r_2$ given by
 \[
    \frac{1}{q_2} := \frac{1}{(4\si+1)q_0'} =
    \frac{4(\si+1)-d\si}{4(\si+1)(4\si+1)}\quad \text{and}\quad
    \frac{1}{r_2} :=\frac{1}{2}-\frac{2}{dq_2}.
  \]
We check that for $d\le 5$, $1/q_2>0$, and $1/q_2<1/2$ since $\si\ge
2/d$: $q_2\in (2,\infty)$, and if $d=1$, $q_2>4$, hence, by definition
of $r_2$, $(q_2,r_2)$ is admissible. Sobolev embedding reads
\begin{equation*}
  W^{s,r_2}\hookrightarrow L^{r_1},\quad
  \frac{1}{r_1}+\frac{s}{d}= \frac{1}{r_2}\Longleftrightarrow
  s=\frac{d}{2} - \frac{d+6}{8\si+2}.
\end{equation*}
We note that for $\frac{2}{d}\le \si<\frac{2}{(d-2)_+}$, $s\in (0,1)$,
hence \eqref{eq:source-step1} entails
\begin{align*}
  \left\| |\Pi_{\tau} u|^{4\si+1} \right\|_{L^{q_0'} (I; L^{r_0'})} &
   \lesssim\tau^{-1/2} \| u \|_{L^{q_2} (I; L^{r_1})}^{4\si+1}\lesssim \tau^{-1/2} \| u \|_{L^{q_2} (I; W^{s,r_2})}^{4\si+1}\\
 &\lesssim \tau^{-1/2} \| u \|_{L^{q_2} (I;
   W^{1,r_2})}^{4\si+1}\lesssim \tau^{-1/2}\|u\|^{4\si+1}_{Y(I)},
\end{align*}
and the first estimate of the lemma follows.
\smallbreak

For the second term of left hand side of the lemma,
 H\"older inequality yields:
\begin{align*}
    \left\| |\Pi_{\tau} u|^{2\si} \Pi_{\tau} (|u|^{2\si} u) \right\|_{L^{q_0'} (I; L^{r_0'})}
    &\le \left\| \Pi_{\tau} u \right\|_{L^{(4\si+1)q_0'} (I;
      L^{(4\si+1)r_0'})}^{2\si}\\
    &\quad\times
        \left\| \Pi_{\tau} (|u|^{2\si} u)
      \right\|_{L^{\frac{4\si+1}{2\si+1}q_0'} (I;
      L^{\frac{4\si+1}{2\si+1}r_0'})}.
\end{align*}
In view of Lemma~\ref{lem:bernstein}, again since $r_1 <(4\si+1)r_0'$,
\begin{equation*}
  \left\| \Pi_{\tau} u \right\|_{L^{(4\si+1)q_0'} (I;
      L^{(4\si+1)r_0'})} \lesssim
    \tau^{\frac{d}{2}\(\frac{1}{(4\si+1)r_0'}-\frac{1}{r_1}\)}\left\|u
    \right\|_{L^{(4\si+1)q_0'} (I;  L^{r_1})} .
  \end{equation*}
  Still from Lemma~\ref{lem:bernstein},
  \begin{equation*}
    \left\| \Pi_{\tau} (|u|^{2\si} u)
      \right\|_{L^{\frac{4\si+1}{2\si+1}q_0'} (I;
      L^{\frac{4\si+1}{2\si+1}r_0'})}\lesssim
    \tau^{\frac{d}{2}\(\frac{2\si+1}{(4\si+1)r_0'}-\frac{2\si+1}{r_1}\)}
    \left\| |u|^{2\si} u
      \right\|_{L^{\frac{4\si+1}{2\si+1}q_0'} (I;
      L^{\frac{r_1}{2\si+1}})}.
  \end{equation*}
  Therefore,
  \begin{align*}
    \left\| |\Pi_{\tau} u|^{2\si} \Pi_{\tau} (|u|^{2\si} u)
    \right\|_{L^{q_0'} (I; L^{r_0'})}&
  \lesssim \tau^{\frac{d}{2} \(\frac{1}{r_0'}-\frac{4\si+1}{r_1}\)}\| u
                                       \|_{L^{(4\si+1)q_0'} (I; L^{r_1})}^{2\si} \\
    &\quad\times
        \left\| |u|^{2\si} u \right\|_{L^{\frac{4\si+1}{2\si+1}q_0'} (I; L^{\frac{r_1}{2\si+1}})} \\
    &\lesssim\tau^{-1/2} \| u \|_{L^{(4\si+1)q_0'} (I;
      L^{r_1})}^{4\si+1}= \tau^{-1/2} \| u \|_{L^{q_2} (I;
      L^{r_1})}^{4\si+1},
  \end{align*}
where we have used H\"older inequality for the last estimate, and we
are back to the situation of the first case.
\end{proof}

%%%%%%%%%%%%%%%%%%%%%%%%%%%%%%%%%%%%%%%%%%%%%%%%%%%%%%%%%%%%%%%%%%%%%%%%%%%%%%%%%%%%%%%%%%%%%%%%

\section{Preliminary estimates: the numerical solution}
\label{sec:numerical}

The discrete (in time) counterpart of the $L^q_tL^r$ norms involved
in the standard Strichartz estimates is:
\begin{equation*}
      \|u\|_{\ell^q (n\tau \in I;\, L^r )}
    = \Big( \tau \sum_{n\tau \in I} \|u(n \tau)\|_{L^r }^q \Big)^{1/q}.
\end{equation*}
For $a,n\in \N$, $a+1\le n$, the discrete Duhamel formula associated
to $Z_\tau$ reads:
\begin{equation}\label{eq:discrete-duhamel}
  \begin{aligned}
    Z_{\tau}(n \tau) &= S_{\tau}((n-a) \tau)Z_\tau(a\tau)+ \tau
\sum_{k=a}^{n-1} S_{\tau}(n\tau -k \tau) \frac{N(\tau) -
  \Id}{\tau} Z_{\tau}(k\tau),
\end{aligned}
\end{equation}
which is compared to the formula \eqref{eq:duhamel}. We now list
analogues of Proposition~\ref{prop:strichartz} in the discrete (in
time) setting, a framework which is not quite as classical.

\begin{proposition}[{\cite[Theorem 2.1]{Ignat2011}}]\label{prop:strichartz-ignat}
Let $(q,r)$,  $(q_1,r_1)$ and $(q_2, r_2)$ be any admissible
pairs. Then, there exist $C_{d,q}, C_{d,q_1,q_2}>0$ such that
    \begin{equation}\label{eq:discr-stri-homo}
    \|S_{\tau}(\cdot) \phi \|_{\ell^q (\tau \Z; L^r )}
    \le C_{d,q}  \|\phi \|_{L^2 },
    \end{equation}
and
    \begin{equation}\label{eq:disc-stri-inhom}
    \left\| \tau \sum_{k=-\infty}^{n-1}S_{\tau}\((n-k) \tau\) f (k\tau)\right\|_{\ell^{q_1} (\tau \Z; L^{r_1} )}
    \le C_{d,q_1,q_2} \|f\|_{\ell^{q_2'}(\tau \Z; L^{r_2'} )}
    \end{equation}
hold for all $\phi\in L^2 $ and $f\in \ell^{q_2'}(\tau
\Z; L^{r_2'})$.
\end{proposition}
The following version of Bernstein lemma is
borrowed from \cite{ChoiKoh2021}:

\begin{lemma}{\cite[Lemma~2.6]{ChoiKoh2021}}\label{lem:bernstein}
For any $1\le q \le r <\infty$ and $\phi:\R^d\rightarrow\mathbb{C}$, we have
    \begin{equation}\label{eq-2-20}
    \big\| \Pi_{\tau}\phi - \phi \big\|_{L^{r} }
    \le C\tau^{1/2} \big\| (-\Delta)^{1/2} \phi \big\|_{L^{r} } ,
    \end{equation}
    \begin{equation}\label{eq-2-20'}
    \| \Pi_{\tau} \phi\|_{L^r }
    \le C \|\phi\|_{L^r } ,
    \end{equation}
\begin{equation}\label{eq-2-21}
    \big\| \nabla (\Pi_{\tau} \phi) \big\|_{L^r }
    \le C \tau^{-\frac{1}{2}} \| \phi \|_{L^r },
    \end{equation}
    and
    \begin{equation}\label{eq-2-22}
    \|\Pi_{\tau} \phi \|_{L^{r}}
    \le C \tau^{\frac{d}{2} \left( \frac{1}{r}-\frac{1}{q}\right)} \| \phi\|_{L^{q}}.
    \end{equation}
\end{lemma}

An important technical novelty compared both with previous studies of
\eqref{eq:NLS},
and with error estimates for time discretization, is that since we
consider the vectorfield $J$ in the presence of the truncated
propagator $S_\tau$, we face a lack of commutation: $\Pi_\tau$
commutes with $\Id$ and $\nabla$, but not with $J$. This lack of
commutation turns out to be controlled thanks to the following lemma:
\begin{lemma}\label{lem:cutoff}
  For any $1< p<\infty$, there  exists $C_p>0$ such that
  \begin{equation}
    \label{eq:cutoffJ}
    \left\|J(t)\Pi_\tau\phi- \Pi_\tau J(t)\phi\right\|_{L^p}\le
  C_p\tau^{1/2} \|\phi\|_{L^p},\quad \forall \phi\in\Sigma,\quad\forall t\in \R.
  \end{equation}
\end{lemma}
\begin{proof}
  The Fourier transform of $\Pi_\tau J(t)\phi$ is given by
  \begin{align*}
    \F \( \Pi_\tau J(t)\phi\)(\xi) &= \chi\(\tau^{1/2}\xi\)\( i\nabla
                                     \hat \phi(\xi) -t\xi \hat \phi(\xi)\)\\
    & = i\nabla \big( \chi\big(\tau^{1/2}\xi\big)\hat \phi(\xi)\big) -t\xi
      \chi\big(\tau^{1/2}\xi\big) \hat \phi(\xi)-i\tau^{1/2} (\nabla
      \chi)\big(\tau^{1/2}\xi\big)\hat \phi(\xi) \\
    & =\F\( J(t)\Pi_\tau \phi\) - i\tau^{1/2} (\nabla
      \chi)\(\tau^{1/2}\xi\)\hat \phi(\xi) .
  \end{align*}
  The lemma follows from the boundedness of $\nabla \chi$ and basic Fourier
  multiplier theory, provided that $\nabla \chi\in C^k$ with $k>d/2$
  (see e.g. \cite[p.~96]{Stein70}).
\end{proof}

Since the operators $S_{\tau}$ and $\nabla$ commute, the following
corollary is immediate in the case $A\in \{\Id,\nabla\}$, but requires
more work in the case $A=J$.

\begin{corollary}\label{cor:stri-disc}
Let $(q,r)$, $(q_1,r_1)$ and $(q_2,r_2)$ be any admissible
pairs. Then, there exist $C_{d,q}, C_{d,q_1,q_2}>0$ such that for any
$t_0\in \tau\Z$ and
$A\in \{\Id,\nabla\}$,
    \begin{equation}\label{eq-st-7}
    \|AS_{\tau}(\cdot-t_0) \phi \|_{\ell^q (\tau \Z; L^r)}
    \le C_{d,q}  \|A\phi \|_{L^2},
    \end{equation}
and
    \begin{equation}\label{eq-st-8}
    \left\| \tau A\sum_{k=-\infty}^{n-1}S_{\tau}((n-k) \tau) f (k\tau)\right\|_{\ell^{q_1} (\tau \Z ;L^{r_1})}
    \le C_{d,q_1,q_2} \|A f\|_{\ell^{q_2'}(\tau \Z; L^{r_2'} )}
    \end{equation}
hold for all $\phi\in \Sigma$ and  $f$ such that $Af\in
\ell^{q_2'}(\tau \Z; L^{r_2'})$.
While for $A=J$, we have
\begin{align}
&\|J(\cdot)S_{\tau}(\cdot-t_0) \phi \|_{\ell^q (\tau \Z; L^r)}
    \le C_{d,q} \(\tau^{1/2}\|\phi\|_{L^2}+\|J(t_0)\phi \|_{L^2}\),\label{Jineq1}\\
&\left\| \tau J(n\tau)\sum_{k=-\infty}^{n-1}S_{\tau}((n-k) \tau) f (k\tau)\right\|_{\ell^{q_1} (\tau \Z ;L^{r_1})}\nonumber\\
&\quad    \le C_{d,q_1,q_2} \(\tau^{1/2}\|f\|_{\ell^{q_2'}(\tau \Z; L^{r_2'} )}+ \|J f\|_{\ell^{q_2'}(\tau \Z; L^{r_2'} )}\).\label{Jineq2}
    \end{align}
\end{corollary}
\begin{proof}
For $A=\{\Id, \nabla\}$, \eqref{eq-st-7} and \eqref{eq-st-8} follow
directly from Proposition \ref{prop:strichartz-ignat} by
noticing that $A$ commutes with $S_\tau$. On the other hand, $J$ does
not commute with $S_\tau$.  In view of the properties of $\chi$, we
have the identity
\begin{equation*}
  \Pi_\tau = \Pi_\tau\Pi_{\tau/4},\quad \forall \tau>0.
\end{equation*}
Therefore, using the standard notation $[A,B]=AB-BA$,
\begin{align*}
J(t)S_\tau(t-t_0)\phi&=J(t)S_\tau(t-t_0)\Pi_{\tau/4} \phi\\
&=[J(t),\Pi_\tau]S(t-t_0)\Pi_{\tau/4}\phi +\Pi_\tau J(t)S(t-t_0)\Pi_{\tau/4}\phi\\
&=[J(t),\Pi_\tau]S_{\tau/4}(t-t_0)\phi +\Pi_\tau J(t)S(t-t_0)\Pi_{\tau/4}\phi.
\end{align*}
Since
$J(t)=S(t)xS(-t)=S(t-t_0)J(t_0)S(t_0-t)$,  the last term is equal to
\begin{align*}
  \Pi_\tau J(t)S(t-t_0)\Pi_{\tau/4}\phi& = \Pi_\tau
 S(t-t_0)J(t_0)\Pi_{\tau/4}\phi= S_\tau (t-t_0) J(t_0)\Pi_{\tau/4}\phi\\
& = S_\tau (t-t_0) [J(t_0),\Pi_{\tau/4}]\phi + S_\tau (t-t_0) \Pi_{\tau/4}J(t_0)\phi\\
& = S_\tau (t-t_0) [J(t_0),\Pi_{\tau/4}]\phi + S_\tau (t-t_0) J(t_0)\phi.
\end{align*}
We infer
\begin{align*}
&\|J(\cdot)S_{\tau}(\cdot-t_0) \phi \|_{\ell^q (\tau \Z; L^r)}\\
&\quad\le \|S_\tau(\cdot-t_0)J(t_0)\phi\|_{\ell^q (\tau \Z; L^r)}+
\|[J(\cdot),\Pi_\tau] S_{\tau/4}(\cdot -t_0)\phi\|_{\ell^q (\tau \Z; L^r)}\\
&\qquad+\|S_\tau (\cdot -t_0) [J(t_0),\Pi_{\tau/4}]\phi\|_{\ell^q (\tau \Z; L^r)} \\
&\quad\le C_{d,q}\|J(t_0)\phi\|_{L^2}+C\tau^{1/2}
\|S_{\tau/4}(\cdot-t_0)\phi\|_{\ell^q (\tau \Z; L^r)}+ C_{d,q}\|[J(t_0),\Pi_{\tau/4}]\phi\|_{L^2}\\
&\quad\le \tilde C_{d, q}\(\tau^{1/2}\|\phi\|_{L^2}+\|J(t_0)\phi\|_{L^2}\),
\end{align*}
where we have used \eqref{eq:discr-stri-homo} and \eqref{eq:cutoffJ}.
Similarly, \eqref{Jineq2} can be established by applying \eqref{eq:disc-stri-inhom} and \eqref{eq:cutoffJ}.
\end{proof}
Propositions~\ref{prop:strichartz} and \ref{prop:strichartz-ignat} and
Christ-Kiselev lemma imply this slight extension (to incorporate $J$)
of \cite[Lemma~4.5]{Ignat2011}:

\begin{corollary}[{\cite[Lemma 4.5]{Ignat2011}}]
For any admissible pairs $(q_1,r_1)$ and $(q_2,r_2)$, we have, for
$A\in \{\Id,\nabla\}$,
    \begin{equation}\label{eq-2-6}
    \left\| A\int_{s < n \tau} S_{\tau}(n \tau -s) f(s) ds
    \right\|_{\ell^{q_1} (\tau \Z; L^{r_1} )}
    \le C_{d,q_1,q_2} \left\| Af\right\|_{L^{q_2'} (\R; L^{r_2'})},
    \end{equation}
and
    \begin{equation}\label{eq-2-6bis}
      \begin{aligned}
            \left\| J(n\tau)\int_{s < n \tau} S_{\tau}(n \tau -s) f(s) ds
    \right\|_{\ell^{q_1} (\tau \Z; L^{r_1} )}
    &\le \\
\le C_{d,q_1,q_2} \Big(\tau^{1/2}\|  f\|_{L^{q_2'} (\R; L^{r_2'})}&
    +\left\| J f\right\|_{L^{q_2'} (\R; L^{r_2'})}\Big).
      \end{aligned}
    \end{equation}
\end{corollary}

\begin{lemma}{\cite[Lemma~2.5]{ChoiKoh2021}}\label{lem:propNL}
There exists  $c>0$ such that
    \begin{equation}\label{eq:NLlip}
    \left| \frac{N(\tau) - \Id}{\tau} v - \frac{N(\tau) - \Id}{\tau}w \right|
    \le c\(|v|^{2\si} + |w|^{2\si}\)|v-w|
    \end{equation}
and
    \begin{equation}\label{eq:NLpuissance}
    \left| \frac{N(\tau)-\Id}{\tau} v\right|
    = \left| \frac{\exp(-i \tau \lambda|v|^{2\si}) -1}{\tau} v\right| \le |v|^{2\si+1}
    \end{equation}
hold for all $v, w \in \mathbb{C}$.
Furthermore, for weakly differentiable $f : \R^d\to \C$, we have the pointwise estimate
    \begin{equation}\label{eq:gradientNL}
    \left| \nabla \left( \frac{N(\tau) - \Id}{\tau} f \right) \right| \le (2\si+1) |f|^{2\si} |\nabla f|.
    \end{equation}
\end{lemma}

\begin{lemma}{\cite[Lemma~2.7]{ChoiKoh2021}}\label{lem-5-2}
For any admissible pairs $(q_1,r_1)$ and $(q_2,r_2)$, there is a
constant $C_{d,q_1,q_2}>0$ such that
    \begin{equation}\label{eq-5-5}
    \begin{split}
    &\left\| \int_{s< n \tau} S_{\tau}(n \tau- s) f(s) ds
        - \tau \sum_{k=-\infty}^{n-1} S_{\tau}(n \tau - k \tau) f(k
        \tau) \right\|_{\ell^{q_1} (\tau \Z; L^{r_1} )} \\
    &\qquad\le C_{d,q_1,q_2}~ \tau^{1/2} \|f\|_{L^{q_2'}(\R; W^{1,r_2'})}
        + C_{d,q_1,q_2}~ \tau \| \partial_t f \|_{L^{q_2'}(\R; L^{r_2'})}
    \end{split}
    \end{equation}
hold for any test function $f \in \mathcal{S}(\R^{d+1})$.
\end{lemma}

%%%%%%%%%%%%%%%%%%%%%%%%%%%%%%%%%%%%%%%%%%%%%%%%%%%%%%%%%%%%%%%%%%%%%%%%%%%%%%%%%%%%%%%%%%%%%%%%

\section{$L^2$ convergence}\label{sec:CVL2}

\begin{proof}[Proof of Theorem~\ref{theo:CVL2}]
In view of Theorem~\ref{theo:aprioriX} and the assumption of Theorem~\ref{theo:CVL2}, we see that $u$ and $Z_{\tau}$ satisfy the following estimates
  \begin{equation}\label{eq:aprioriCVL2}
	\begin{aligned}
	&\|u\|_{Y(\R_+)}
	\le M_1, \\
	\max_{A\in \{\Id,\nabla,J\}}&\| A(n\tau)Z_{\tau}(n \tau)\|_{\ell^{\infty} (\tau\N; L^2)} +
	\| Z_{\tau}(n \tau)\|_{\ell^{q_0} (\tau\N; W^{1,r_0} )}
	\le M_2.
	\end{aligned}
\end{equation}
We shall estimate $Z_{\tau}(n \tau) - \Pi_{\tau} u(n \tau)$ instead of
$Z_{\tau}(n \tau) - u(n \tau)$, in view of
    \[
    \left\| u(n \tau) - \Pi_{\tau} u(n \tau) \right\|_{\ell^{\infty} (I; L^2)}
    \le C\tau^{1/2} \| u(n \tau)\|_{\ell^{\infty} (I; H^1)}
    \le C\tau^{1/2} M_1,
    \]
    by \eqref{eq-2-20}.
We recall that $\gamma<\infty$ is defined by \eqref{eq:gamma}. We recall that
$\gamma\delta(r_0)>1$ (see \eqref{eq:gammadeltar0}).
Therefore, in view of \eqref{eq:ellqLr}, \eqref{eq:aprioriCVL2} and
\eqref{eq:GNlibre}, for any $\eta>0$, we can find a \emph{finite} number
$K=K(\eta)$ of time intervals $I_j=\overline{[m_j,m_{j+1})}$
with $m_j\in \N$ if $j\le K$, $m_{K+1}=\infty$, and $\tau(\eta)>0$ such that for $1\le j\le K$ and
$0<\tau\le \tau(\eta)$,
\begin{equation*}
   \|u(k\tau)\|_{\ell^{\gamma}(\tau I_j;L^{r_0})}+ \|Z_\tau(k\tau)\|_{\ell^{\gamma}(\tau
    I_j;L^{r_0})}\le \eta,\quad \R_+ = \bigcup_{j=1}^K I_j.
\end{equation*}
We consider the adherence of $I_j$ because unlike in the continuous
case, for $j+1\le K$, the singleton $\{m_{j+1}\}$ is not of measure
zero, and actually becomes important in the following discussion.
On each interval $I_j$, the discrete Duhamel's formula
\eqref{eq:discrete-duhamel} can be written as
    \begin{equation}\label{eq-a-1'}
    Z_{\tau}(m_j\tau +n\tau)
    =S_{\tau}(n\tau) Z_{\tau}(m_j\tau) +
    \tau \sum_{k=0}^{n-1} S_{\tau}(n \tau -k \tau) \frac{N (\tau) -\Id}{\tau} Z_{\tau}(m_j\tau + k \tau),
  \end{equation}
 for $0\le n< m_{j+1}-m_j$ ($m_{j+1}$ may be infinite, but $m_j$ is
 always finite).
By combining this with \eqref{eq:duhamel}, we obtain the following decomposition:
 \begin{equation}\label{eq:decomp}
    Z_{\tau}(m_j \tau +n \tau) - \Pi_{\tau} u(m_j \tau +n \tau)
    = \mathcal{A}_1(n) + \mathcal{A}_2(n) + \mathcal{A}_3(n) + \mathcal{A}_4(n) ,
 \end{equation}
where
    \begin{align*}
    \mathcal{A}_1(n) &:= S_{\tau} (n\tau)\( Z_{\tau}(m_j \tau) - \Pi_{\tau} u(m_j \tau)\),    \\
    \mathcal{A}_2(n) &:= S_{\tau}(n\tau) \( \Pi_{\tau} u (m_j \tau)-u(m_j \tau)\),
    \\
    \mathcal{A}_3(n)& := \tau \sum_{k=0}^{n-1} S_{\tau}(n \tau -k \tau) \Big( \frac{N(\tau) -\Id}{\tau} Z_{\tau}(m_j\tau+k\tau)
        - \frac{N(\tau) -\Id}{\tau} \Pi_{\tau} u(m_j\tau+k\tau) \Big),
    \\
    \mathcal{A}_4(n)&:= \tau \sum_{k=0}^{n-1} S_{\tau}(n \tau -k
      \tau) \frac{N(\tau) -\Id}{\tau} \Pi_{\tau} u(m_j\tau+k\tau)\\
      &\quad
        +i\int_0^{n\tau} S_{\tau}(n\tau -s) \(|u|^{2\si} u\) (m_j\tau+s) ds,
    \end{align*}
    and we omit the dependence of the $\mathcal A_k$'s upon $j$ to
    ease notations.
    The terms $\mathcal A_1$ and $\mathcal A_2$ are linear, while
    $\mathcal A_3$ and $\mathcal A_4$ are nonlinear. The goal is to
    show that in the estimates, the term $\mathcal A_3$ can be absorbed by
    the left hand side of \eqref{eq:decomp}, $\mathcal A_2$ and
    $\mathcal A_4$ are
    $\O(\tau^{1/2})$, and $\mathcal A_1$ is then estimated by
    induction on $j$.

    Let
    $(q,r) \in \{(q_0, r_0), (\infty,2)\}$.
 The homogeneous Strichartz estimate \eqref{eq:discr-stri-homo} yields
    \begin{equation*}
    \| \mathcal{A}_1\|_{\ell^{q}(\tau I_j; L^r)}
    \le C_{d,q} \left\| Z_{\tau}(m_j \tau) - \Pi_{\tau} u (m_j \tau) \right\|_{L^2}.
    \end{equation*}
The second term is controlled via \eqref{eq:discr-stri-homo} and
\eqref{eq-2-20}, by
    \begin{equation*}
    \begin{aligned}
    \| \mathcal{A}_2\|_{\ell^{q}(\tau I_j; L^r)}
    &= \left\|S_{\tau}(n\tau) \( \Pi_{\tau} u(m_j \tau) - u (m_j
      \tau)\)\right \|_{\ell^q (\tau I_j; L^r)} \\
    &\le C_{d,q} \left\| \Pi_{\tau} u(m_j \tau) - u (m_j \tau) \right\|_{L^2} \\
    &\le C_{d,q} \tau^{1/2} \left\| (-\Delta)^{1/2} u(m_j \tau) \right\|_{L^2}
    \le C_{d,q} \tau^{1/2} M_1,
    \end{aligned}
  \end{equation*}
  where we have used \eqref{eq:aprioriCVL2}.
To estimate $\mathcal{A}_3$, we use \eqref{eq:NLlip} to find
    \begin{equation*}
    \left| \frac{N(\tau) -\Id}{\tau} Z_{\tau} - \frac{N(\tau) -\Id}{\tau} \Pi_{\tau} u \right|
    \lesssim \( |Z_{\tau}|^{2\si} +|\Pi_{\tau}u|^{2\si}\) |Z_{\tau} - \Pi_{\tau} u| .
    \end{equation*}
The inhomogeneous Strichartz estimate \eqref{eq:disc-stri-inhom} and
H\"older inequality \eqref{eq:holder} yield
    \begin{equation*}
    \begin{aligned}
    \| \mathcal{A}_3\|_{\ell^{q}(\tau I_j; L^r)}
    &\lesssim  \left\|\( |Z_{\tau}|^{2\si} +|\Pi_{\tau}u|^{2\si} \)
      |Z_{\tau} - \Pi_{\tau} u|  \right\|_{\ell^{q_0'} (\tau I_j; L^{r_0'} )} \\
    &\lesssim \( \|Z_{\tau}\|_{\ell^{\gamma}(\tau I_j;L^{r_0})}^{2\si} +
    \|\Pi_{\tau}u\|_{\ell^{\gamma}(\tau I_j;L^{r_0})}^{2\si}  \) \|Z_{\tau}
    - \Pi_{\tau} u\|_{\ell^{q_0}(\tau I_j;L^{r_0})} \\
    &\le \mathbf C_{d, \sigma,q}\eta^{2\si}   \|Z_{\tau}
    - \Pi_{\tau} u\|_{\ell^{q_0}(\tau I_j;L^{r_0})} ,
    \end{aligned}
  \end{equation*}
  for some constant $\mathbf C_{d, \sigma,q}$, where we have used the definition of the intervals $I_j$ in terms of
  $\eta$. We now choose $\eta>0$ sufficiently small so that
  $\mathbf C_{d, \sigma,q_0}\eta^{2\si}<1/2$, thus
  \begin{equation*}
   \| \mathcal{A}_3\|_{\ell^{q_0}(\tau I_j; L^{r_0})}\le \frac{1}{2} \|Z_{\tau}
    - \Pi_{\tau} u\|_{\ell^{q_0}(\tau I_j;L^{r_0})},\quad 1\le j\le K.
  \end{equation*}
The estimate of $\mathcal A_4$ is postponed to Lemma~\ref{lem:A4}
below,
\begin{equation}\label{eq:A4}
  \max_{1\le j\le K}\( \|\mathcal A_4\|_{\ell^\infty (\tau I_j;L^2)}+
  \|\mathcal A_4\|_{\ell^{q_0} (\tau I_j;L^{r_0})} \)\lesssim \tau^{1/2}.
\end{equation}
Gathering all the previous estimates, we first get, with
$(q,r)=(q_0,r_0)$,
\begin{align*}
  \|Z_{\tau}
    - \Pi_{\tau} u\|_{\ell^{q_0} (\tau I_j;L^{r_0})} &\le  C\left\|
      Z_{\tau}(m_j \tau) - \Pi_{\tau} u (m_j \tau) \right\|_{L^2}\\
      &\quad +
    C\tau^{1/2} + \frac{1}{2} \|Z_{\tau}
    - \Pi_{\tau} u\|_{\ell^{q_0}(\tau I_j;L^{r_0})},
\end{align*}
hence
\begin{equation*}
  \|Z_{\tau}
    - \Pi_{\tau} u\|_{\ell^{q_0}(\tau I_j;L^{r_0})} \le 2 C\left\|
      Z_{\tau}(m_j \tau) - \Pi_{\tau} u (m_j \tau) \right\|_{L^2}
    +2C\tau^{1/2} ,\quad 1\le j\le K.
  \end{equation*}
  Taking now $(q,r)=(\infty,2)$, we obtain
  \begin{align*}
  \|Z_{\tau}
    - \Pi_{\tau} u\|_{\ell^{\infty}(\tau I_j;L^{2})} &\lesssim \left\|
      Z_{\tau}(m_j \tau) - \Pi_{\tau} u (m_j \tau) \right\|_{L^2} + \tau^{1/2} + \|Z_{\tau}
                                                  - \Pi_{\tau} u\|_{\ell^{q_0} (\tau I_j;L^{r_0})}\\
    &\lesssim
  \left\|
      Z_{\tau}(m_j \tau) - \Pi_{\tau} u (m_j \tau) \right\|_{L^2} +
    \tau^{1/2},\quad 1\le j\le K.
  \end{align*}
  Now by construction $m_1=0$, $ Z_{\tau}(m_1\tau) - \Pi_{\tau} u
  (m_1 \tau) = 0$, and for $2\le j\le K$,
  \begin{equation*}
    \left\|
      Z_{\tau}(m_j \tau) - \Pi_{\tau} u (m_j \tau) \right\|_{L^2} \le \left\|
      Z_{\tau} - \Pi_{\tau} u  \right\|_{\ell^\infty(\tau I_{j-1};L^2)} ,
  \end{equation*}
we have the first conclusion in Theorem~\ref{theo:CVL2}. The second one
is then a direct consequence of this first point and the scattering
result recalled in Theorem~\ref{theo:aprioriX}: there exists $u_+\in
\Sigma$ such that
\begin{equation*}
  \|u(t)-S(t)u_+\|_{L^2}\Tend t \infty 0,
\end{equation*}
where we have used the fact that $S(t)$ is unitary on $L^2$.
\end{proof}

We conclude this section by proving \eqref{eq:A4}.

\begin{lemma}\label{lem:A4}
  For $u\in Y(\R_+)$, $\tau \in (0,1)$, denote
  \[\mathcal{A}(u)(n\tau)=\tau \sum_{k=0}^{n-1} S_{\tau}(n \tau -k \tau) \frac{N(\tau)-\Id}{\tau} \Pi_{\tau} u (k \tau)
        + i\int_0^{n\tau} S_{\tau}(n\tau -s) |u|^{2\si} u (s) ds.\]
Then for all admissible pairs $(q,r)$,
    \[\left\|  \mathcal{A}(u)\right\|_{\ell^q (\tau\N; L^r)} \lesssim \tau^{1/2}\(\|u\|_{Y(\R_+)}^{2\si+1}+\|u\|_{Y(\R_+)}^{4\si+1}\)  .
     \]
\end{lemma}

\begin{proof}
  As observed in the previous computations, we recall that the
  assumptions of Lemma~\ref{lem:A4} imply
  \begin{equation*}
    u\in L^\gamma(\R_+;L^{r_0}),
  \end{equation*}
  where $\gamma$ is given by \eqref{eq:gamma}, since
  $\gamma\delta(r_0)>1$. Decompose $\mathcal{A}(u)(n\tau)$ as
  \[\mathcal{A}(u)(n\tau)=\mathcal{A}_1(u)(n\tau)+\mathcal{A}_2(u)(n\tau),\]
  where
    \begin{equation*}
    \begin{aligned}
    \mathcal{A}_1(u)(n\tau)&= \tau \sum_{k=0}^{n-1} S_{\tau}(n \tau -k \tau) \mathcal{B}_1(u)(k\tau)
        - \int_0^{n\tau} S_{\tau}(n\tau -s) \mathcal{B}_1(u)(s) ds, \\
  \mathcal{A}_2(u)(n\tau)&=\int_0^{n\tau} S_{\tau}(n\tau -s) \mathcal{B}_1(u)(s)ds+ i\int_0^{n\tau} S_{\tau}(n\tau -s) |u|^{2\si} u (s)ds,
    \end{aligned}
    \end{equation*}
    with
    \[\mathcal{B}_1(u)(s) := \frac{N(\tau) -\Id}{\tau} \Pi_{\tau} u(s).\]
    By Lemma \ref{lem-5-2}, we can estimate
\[      \left\| \mathcal{A}_1(u) \right\|_{\ell^q (\tau\N; L^r)}\le C \tau^{1/2} \left\| \mathcal{B}_1(u) \right\|_{L^{{q_0}'}(\R_+; W^{1,{r_0}'})}
        + C \tau \left\| \mathcal{B}_1(u) \right\|_{W^{1,{q_0}'}(\R_+; L^{{r_0}'})},
\]
for all admissible pair $(q,r)$. Lemma~\ref{lem:propNL} entails
\begin{equation*}
  \left\| \mathcal{B}_1(u) \right\|_{L^{{q_0}'}(\R_+;
  W^{1,{r_0}'})}\lesssim \left\| |\Pi_{\tau} u|^{2\si+1}
                           \right\|_{L^{q_0'}(\R_+; L^{r_0'})}
        + \left\| |\Pi_{\tau} u|^{2\si} |\nabla \Pi_{\tau} u| \right\|_{L^{q_0'}(\R_+; L^{r_0'})}.
      \end{equation*}
  Using now H\"older inequality \eqref{eq:holder} and Lemma~\ref{lem:bernstein},
   \begin{align*}
  \left\| \mathcal{B}_1(u) \right\|_{L^{{q_0}'}(\R_+;
  W^{1,{r_0}'})}& \lesssim \| \Pi_{\tau} u\|^{2\si}_{L^\gamma(\R_+;L^{r_0})}
                  \|\Pi_{\tau} u\|_{L^{q_0}(\R_+;W^{1,r_0})}\\
    & \lesssim \| u\|^{2\si}_{L^\gamma(\R_+;L^{r_0})}
  \| u\|_{L^{q_0}(\R_+;W^{1,r_0})} \lesssim \|u\|_{Y(\R_+)}^{2\si+1}.
   \end{align*}
   For the other term,
   Lemma~\ref{lem:propNL} yields
   \begin{align*}
     \left\| \mathcal{B}_1(u) \right\|_{W^{1,{q_0}'}(\R_+; L^{{r_0}'})}
   &= \left\| \d_t \( \frac{N(\tau) -\Id}{\tau} \Pi_{\tau} u \)
    \right\|_{L^{q_0'}(\R_+; L^{r_0'})} \\
&  \lesssim \left\| |\Pi_{\tau}u|^{2\si} \d_t \Pi_{\tau} u
  \right\|_{L^{q_0'}(\R_+; L^{r_0'})}.
    \end{align*}
    Observing that $\d_t$ and $\Pi_\tau$ commute, \eqref{eq:NLS} implies
    \begin{align*}
     \left\| \mathcal{B}_1(u) \right\|_{W^{1,{q_0}'}(\R_+; L^{{r_0}'})}
   &\lesssim  \left\|  |\Pi_{\tau}u|^{2\si} \Pi_{\tau}\Delta u
     \right\|_{L^{q_0'}(\R_+; L^{r_0'})} \\
      &\quad+ \left\| |\Pi_{\tau}u|^{2\si}  \Pi_{\tau} \(|u|^{2\si}u\)
     \right\|_{L^{q_0'}(\R_+; L^{r_0'})} .
    \end{align*}
  For the first term of the right side, H\"older inequality
  \eqref{eq:holder} and
  Lemma~\ref{lem:bernstein} imply
  \begin{align*}
   \left\|  |\Pi_{\tau}u|^{2\si} \Pi_{\tau}\Delta u
     \right\|_{L^{q_0'}(\R_+; L^{r_0'})} &\le  \left\|
    \Pi_{\tau}u\right\|_{L^{\gamma}(\R_+; L^{r_0})}^{2\si}
    \left\|  \Pi_{\tau} \Delta u \right\|_{L^{q_0}(\R_+; L^{r_0})}\\
  &\lesssim \tau^{-1/2}  \left\|
    \Pi_{\tau}u\right\|_{L^{\gamma}(\R_+; L^{r_0})}^{2\si}
    \left\|  \Pi_{\tau} \nabla u \right\|_{L^{q_0}(\R_+; L^{r_0})}\\
     &\lesssim \tau^{-1/2}  \left\|
   u\right\|_{L^{\gamma}(\R_+; L^{r_0})}^{2\si}
       \left\|  \nabla u \right\|_{L^{q_0}(\R_+; L^{r_0})}\\
    &\lesssim
       \tau^{-1/2}\|u\|_{Y(\R_+)}^{2\si+1}  .
  \end{align*}
  Lemma~\ref{lem:source} provides the following control for the second
  term of the right side:
  \begin{align*}
    \left\| |\Pi_{\tau}u|^{2\si}  \Pi_{\tau} \(|u|^{2\si}u\)
     \right\|_{L^{q_0'}(\R_+; L^{r_0'})} \lesssim \tau^{-1/2}\|u\|_{Y(\R_+)}^{4\si+1},
  \end{align*}
  and we come up with
  \begin{align*}
   \left\| \mathcal{A}_1(u) \right\|_{\ell^q (\tau\N; L^r)}&\lesssim
    \tau^{1/2} \left\| \mathcal{B}_1(u) \right\|_{L^{{q_0}'}(\R_+; W^{1,{r_0}'})}
        + \tau \left\| \mathcal{B}_1(u) \right\|_{W^{1,{q_0}'}(\R_+;
    L^{{r_0}'})}\\
    &\lesssim \tau^{1/2}\big(\|u\|_{Y(\R_+)}^{2\si+1}+\|u\|_{Y(\R_+)}^{4\si+1}\big).
  \end{align*}
   To complete the proof, we perform another decomposition,
    \[
     \mathcal{B}_1(u) +i|u|^{2\si} u= \mathcal{B}_2(u) +i\mathcal{B}_3(u),\]
   where
  \[\mathcal{B}_2(u) := \mathcal{B}_1(u)+i|\Pi_{\tau}u|^{2\si} \Pi_{\tau}u,    \quad
    \mathcal{B}_3(u) := |u|^{2\si} u-|\Pi_{\tau}u|^{2\si} \Pi_{\tau}u.
  \]
The discrete inhomogeneous Strichartz estimates \eqref{eq-2-6} yields,
for  $(q,r)$ admissible,
    \begin{align*}
    &\quad\left\|\mathcal{A}_2(u) \right\|_{\ell^q (\tau \mathbb{N}; L^r)}\\
    &\le \Big\| \int_{0}^{n\tau} S_{\tau}(n \tau -s) \mathcal{B}_2(u) ds \Big\|_{\ell^q (\tau\mathbb{N}; L^r)}
        + \Big\| \int_{0}^{n\tau} S_{\tau}(n \tau -s)
      \mathcal{B}_3(u) ds \Big\|_{\ell^q (\tau\mathbb{N}; L^r)} \\
    &\lesssim \left\| \mathcal{B}_2(u) \right\|_{L^{q_0'} (\R_+; L^{r_0'})}
        + \left\| \mathcal{B}_3(u) \right\|_{L^{q_0'} (\R_+; L^{r_0'})}.
    \end{align*}
Recall that $N(\tau)z= ze^{-i\tau|z|^{2\si}}$, Taylor formula yields the pointwise estimate
\[
    \left| \mathcal{B}_2(u) \right|
    = \Big|\frac{N(\tau) -\Id}{\tau} \Pi_{\tau} u +i |\Pi_{\tau}u|^{2\si} \Pi_{\tau}u \Big|
      \lesssim \tau |\Pi_{\tau}u|^{4\si+1}.
\]
Thus, we have
    \begin{equation*}
    \left\| \mathcal{B}_2(u) \right\|_{L^{q_0'} (\R_+; L^{r_0'})}
    \lesssim \tau \left\| |\Pi_{\tau} u|^{4\si+1} \right\|_{L^{q_0'}
      (\R_+; L^{r_0'})}\lesssim \tau^{1/2}\|u\|_{Y(\R_+)}^{4\si+1},
  \end{equation*}
  where we have used Lemma~\ref{lem:source} for the last estimate.

Finally, H\"older inequality \eqref{eq:holder} yields
    \begin{align*}
    \left\| \mathcal{B}_3(u) \right\|_{L^{q_0'} (\R_+; L^{r_0'})}
    &=\left\| |\Pi_{\tau}u|^{2\si} \Pi_{\tau}u - |u|^{2\si} u
      \right\|_{L^{q_0'} (\R_+; L^{r_0'})} \\
      &\lesssim \(\| \Pi_{\tau}u \|_{L^\gamma(\R_+;L^{r_0})}^{2\si} +
        \|u \|_{L^\gamma(\R_+;L^{r_0})}^{2\si}\)
      \left\| \Pi_{\tau}u - u \right\|_{L^{q_0} (\R_+; L^{r_0})}\\
      &\lesssim \tau^{1/2} \|u\|_{Y(\R_+)}^{2\si}
        \| u \|_{L^{q_0} (\R_+; W^{1,r_0})}
        \lesssim   \tau^{1/2} \|u\|_{Y(\R_+)}^{2\si+1} ,
    \end{align*}
where we have used Lemma~\ref{lem:bernstein}. Lemma~\ref{lem:A4} follows.
\end{proof}

%%%%%%%%%%%%%%%%%%%%%%%%%%%%%%%%%%%%%%%%%%%%%%%%%%%%%%%%%%%%%%%%%%%%%%%%%%%%%%%%%

\section{Local stability}
\label{sec:local}

 We denote
  \begin{equation}\label{eq:chi}
    \|v\|_{\Gamma(I)}
    =\|v\|_{\ell^\infty(I;L^2)}+\|\nabla v\|_{\ell^\infty(I;L^2)}+
    \|J v\|_{\ell^\infty(I;L^2)},
  \end{equation}
  the discrete counterpart of the norm $\|\cdot \|_{X(I)}$ defined in
  \eqref{eq:Xnorm}.
Define
\begin{equation*}
  K_0 := 1+\sup_{\tau\in (0,1)}\sup_{\phi\in L^2}\frac{\|S_\tau(\cdot)
    \phi\|_{\ell^{q_0}(\tau\Z; L^{r_0})}+ \|S_\tau(\cdot) \phi
    \|_{\ell^\infty(\tau\Z;L^2)}}{\|\phi\|_{L^2}},
\end{equation*}
which is finite in view of the discrete homogeneous Strichartz
estimates \eqref{eq:discr-stri-homo}.

Recall that in accordance with \eqref{eq:ellqLr}, we have, for $1\le q<\infty$,
\begin{equation*}
  \left\|
      \<n\tau\>^{-\alpha}\right\|_{\ell^q(I)}=\Big(\tau\sum_{n\tau\in
      I}\<n\tau\>^{-\alpha q}\Big)^{1/q}.
\end{equation*}
\begin{proposition}[Local $\Sigma$ stability]\label{prop:local-stab}
  Let $\phi\in \Sigma$. Consider an
  interval $I$ of the form $I=\overline{[a\tau,b\tau)}$, with $a,b\in \N\cup\{\infty\}$ such
  that $0\le a<b\le \infty$. There exist $\tau_0>0$ and $K_1$
  depending only on $d$
  and $\si$ such that if
  \begin{equation}\label{assump}
    6K_1 \left\|
      \<n\tau\>^{-\delta(r_0)}\right\|_{\ell^\gamma(I)}^{2\si}\le
    \(4K_0\|Z_\tau\|_{X(\{a\tau\})}\)^{-2\si},
 \end{equation}
  then
  \begin{equation}
    \label{eq:loc-stab1}
    \max_{A\in \{\Id,\nabla,J\}} \sum_{(q,r)\in
      \{(q_0,r_0),(\infty,2)\}}\| A Z_{\tau}\|_{\ell^q(I;L^r)}\le 4 K_0 \|Z_\tau\|_{X(\{a\tau\})},
    \quad\forall\tau \in (0,\tau_0),
  \end{equation}
  and for all $(q,r)$ admissible, there exists $K_q$ such that
 \begin{equation}\label{eq:loc-stab2}
    \max_{A\in \{\Id,\nabla,J\}}\| AZ_{\tau}\|_{\ell^q (I; L^{r} )}
    \le K_q\|Z_\tau\|_{X(\{a\tau\})} ,\quad \forall\tau \in (0,\tau_0).
    \end{equation}
\end{proposition}
\begin{remark}\label{rem:series}
  We call the above result  a \emph{local} stability, since we require
  some smallness on
  \begin{equation*}
    \tau\sum_{n=a}^{b}\frac{1}{\<n\tau\>^{\gamma\delta(r_0)}}.
    \end{equation*}
    Obviously, the series is convergent, since $\gamma\delta(r_0)>1$,
    and this smallness means that either $b$ is finite and $b-a$ is
    sufficiently small, or $b=\infty$ and $a$ is sufficiently large, a
    case which corresponds to a local result near $t=\infty$ (in the
    spirit of the construction of wave operators in scattering theory,
    see e.g \cite{CazCourant,Ginibre}).  More precisely, for finite
    $b$, we may use the estimate
    \begin{equation*}
      \tau\sum_{n=a}^{b}\frac{1}{\<n\tau\>^{\gamma\delta(r_0)}}\le
      \tau\sum_{n=a}^{b}1=(b-a)\tau=|I|.
    \end{equation*}
    For $a\ge 1$ (and large), we may use instead,
     \begin{equation*}
      \tau\sum_{n=a}^{b}\frac{1}{\<n\tau\>^{\gamma\delta(r_0)}}\le
      \tau\sum_{n=a}^{\infty}\frac{1}{\(n\tau\)^{\gamma\delta(r_0)}}=
      \tau^{1-\gamma\delta(r_0)}
      \sum_{n=a}^{\infty}\frac{1}{n^{\gamma\delta(r_0)}}\lesssim
       \frac{1}{(a\tau)^{\gamma\delta(r_0)-1}}.
    \end{equation*}
  \end{remark}

\begin{proof}
Consider the set
    \begin{equation*}%\label{eq:Lambda}
    \Lambda = \Big\{ N \in \mathbb{N};
    \max_{A\in \{\Id,\nabla,J\}}\sum_{(q,r)\in
      \{(q_0,r_0),(\infty,2)\}}\hspace{-6mm}\|AZ_{\tau}\|_{\ell^{q} ([a\tau, (N+a)\tau]; L^r)}
    \le 4 K_0 \|Z_\tau\|_{X(\{a\tau\})} \Big\}.
    \end{equation*}
If $\Lambda$ is an infinite set, then the conclusion of proposition
follows easily, from the estimates presented below, based on
Strichartz estimates.

Suppose that $\Lambda$ is a finite set, and let $N_*$ be the largest element of $\Lambda$.
First we verify that the set $\Lambda$ is non-empty.
Indeed, by the definitions of $Z_{\tau}$ and $K_0$, we find that, for
$A\in \{\Id,\nabla\}$,
\begin{align*}
  \tau^{\frac{1}{q_0}} \| A Z_{\tau}(a\tau)\|_{L^{r_0}} + &\|
  AZ_{\tau}(a\tau)\|_{L^2}=\tau^{\frac{1}{q_0}} \|S_{\tau}(0)
AZ_\tau(a\tau)\|_{L^{r_0}}
 + \| S_{\tau}(0) A Z_\tau(a\tau)\|_{L^2} \\
  &\le \|S_{\tau}(\cdot) A Z_\tau(a\tau)\|_{\ell^{q_0}(\tau\Z;
    L^{r_0})} + \| S_{\tau}(\cdot) AZ_\tau(a\tau)\|_{\ell^{\infty} (\tau\Z;L^2)}\\
  &\le (K_0-1) \|AZ_\tau(a\tau)\|_{L^2}\le (K_0-1)\|Z_\tau\|_{X(\{a\tau\})}.
\end{align*}
While for $A=J$, applying Lemma \ref{lem:cutoff}, we get
\begin{align*}
 &\quad  \tau^{\frac{1}{q_0}} \| J(a\tau) Z_{\tau}(a\tau)\|_{L^{r_0}} + \|J(a\tau)Z_{\tau}(a\tau)\|_{L^2}\\
  &\le \tau^{\frac{1}{q_0}} \|S_{\tau}(0)
J(a\tau)Z_\tau(a\tau)\|_{L^{r_0}} + \| S_{\tau}(0)J(a\tau)Z_\tau(a\tau)\|_{L^2} \\
&\quad+C\tau^{1/2}\(\tau^{\frac{1}{q_0}}\|Z_\tau(a\tau)\|_{L^{r_0}}+
\|Z_\tau(a\tau)\|_{L^2} \)\\
  &\le \|S_{\tau}(\cdot) J(a\tau) Z_\tau(a\tau)\|_{\ell^{q_0}(\tau\Z;
    L^{r_0})} + \| S_{\tau}(\cdot) J(a\tau) Z_\tau(a\tau)\|_{\ell^{\infty} (\tau\Z;L^2)}\\
    &\quad+C\tau^{1/2}(K_0-1)\|Z_\tau\|_{X(\{a\tau\})}\\
  &\le (K_0-1) \|J(a\tau) Z_\tau(a\tau)\|_{L^2}+C\tau^{1/2}(K_0-1)\|Z_\tau\|_{X(\{a\tau\})}\\
  &\le (1+C\tau^{1/2}) (K_0-1)\|Z_\tau\|_{X(\{a\tau\})}\\
  &< 4K_0\|Z_\tau\|_{X(\{a\tau\})},
\end{align*}
when $\tau\le \tau_0:=1/4C^2$, where $C$ is the  constant in Lemma \ref{lem:cutoff}. Therefore, $0\in \Lambda$.

For $A\in\{\Id,\nabla\}$,  and $(q,
r)\in\{(q_0, r_0), (\infty, 2)\}$,
    \begin{equation}\label{eq:1756}
    \begin{aligned}
    &\quad \|AZ_\tau(n\tau) \|_{\ell^q(a\le n \le a+N_* +1;L^{r})}\\
    &    \le \left\| A S_{\tau}(n\tau-a\tau)
      Z_\tau(a\tau)\right\|_{\ell^{q} (a\le n \le a+N_* +1; L^{r})}\\
&    \quad +\left\|\tau   A\sum_{k=a}^{n-1} S_{\tau} (n\tau -k\tau)
        \frac{N(\tau) -\Id}{\tau} Z_{\tau}(k \tau) \right\|_{\ell^{q}
        (a +1\le n \le a+N_* +1; L^{r} )} \\
&\le\left\| S_{\tau}(n\tau-a\tau)  AZ_\tau(a\tau)\right\|_{\ell^{q} (a\le n \le a+N_* +1; L^{r})}\\
&    \quad +\left\|\tau  \sum_{k=a}^{n-1} S_{\tau} (n\tau -k\tau)A
        \frac{N(\tau) -\Id}{\tau} Z_{\tau}(k \tau) \right\|_{\ell^{q}
        (a +1\le n \le a+N_* +1; L^{r} )}\\
&\le (K_0-1)\|Z_\tau\|_{X(\{a\tau\})}+Q_1,
    \end{aligned}
    \end{equation}
where the last term $Q_1$ is bounded by applying the
Strichartz estimate \eqref{eq-st-8} as follows:
    \begin{equation*}
    \begin{aligned}
   Q_1 &=\left\|\tau   \sum_{k=a}^{n-1} S_{\tau} (n\tau -k\tau)A
        \frac{N(\tau) -\Id}{\tau} Z_{\tau}(k \tau) \right\|_{\ell^{q}
        (a +1\le n \le a+N_* +1; L^{r} )}\\
    &\lesssim \left\| A\frac{N(\tau) -\Id}{\tau} Z_{\tau}(n \tau)
    \right\|_{\ell^{q_0'} (a\le n \le a+ N_* ; L^{r_0'})}.
    \end{aligned}
    \end{equation*}
To estimate the right hand side, we apply Lemma~\ref{lem:propNL} and
H\"older inequality \eqref{eq:holder}. We omit the information $a\le
n\le a+N_*$ to ease notations:
\begin{align*}
  \left\| A\frac{N(\tau) -\Id}{\tau} Z_{\tau}
    \right\|_{\ell^{q_0'}  L^{r_0'}}
  &\lesssim \|Z_{\tau}\|_{\ell^\gamma
    L^{r_0}}^{2\si}\|A Z_{\tau}\|_{\ell^{q_0}L^{r_0}}\\
  &\lesssim \left\|
    \<n\tau\>^{-\delta(r_0)}\right\|_{\ell^\gamma(a\le n\le a+N^*)}^{2\si}
    \|Z_{\tau}\|_{\Gamma([a\tau,(a+N_*)\tau])}^{2\si} \|AZ_{\tau}\|_{\ell^{q_0}L^{r_0}},
\end{align*}
where we have used \eqref{eq:ellqLr}, and (weighted)
Gagliardo-Nirenberg inequality for the
last estimate. Since $N_*\in \Lambda$, we infer that
\begin{align*}
  \left\| A\frac{N(\tau) -\Id}{\tau} Z_{\tau}
    \right\|_{\ell^{q_0'}  L^{r_0'}}\lesssim \left\|
    \<n\tau\>^{-\delta(r_0)}\right\|_{\ell^\gamma(a\le n\le
  a+N^*)}^{2\si} \(4K_0\|Z_\tau\|_{X(\{a\tau\})}\)^{2\si+1}.
\end{align*}
We infer that there exists $K_1$ depending only on $d$ and $\si$ such
that, for any $A\in \{\Id,\nabla\}$,
\begin{align*}
 \sum_{(q,r)\in \{(q_0,r_0),(\infty,2)\}} \|AZ_\tau &\|_{\ell^{q}(a\le n\le
  a+N_*+1;L^{r})}
  \le 2K_0 \|Z_\tau\|_{X(\{a\tau\})} \\
                  &+ K_1 \left\|
    \<n\tau\>^{-\delta(r_0)}\right\|_{\ell^\gamma(a\le n\le
  a+N_*)}^{2\si} \(4K_0\|Z_\tau\|_{X(\{a\tau\})}\)^{2\si+1}.
\end{align*}
We conclude with the case $A=J$. We have from \eqref{eq:discrete-duhamel},
\begin{align*}  J(n\tau)Z_{\tau}(n \tau)&=  J(n\tau)S_{\tau}((n-a) \tau) Z_\tau(a\tau)
 \\
  &\quad + \tau\sum_{k=a}^{n-1} J(n\tau)S_{\tau}(n\tau-k \tau)
  \frac{N(\tau) -\Id}{\tau}
  Z_{\tau}(k \tau).
\end{align*}
We proceed like in the proof of Corollary~\ref{cor:stri-disc}:
$J(t)=S(t-t_0)J(t_0)S(t_0-t),$
hence
\begin{align}
  J(n\tau)S_{\tau}((n-a) \tau) & = J(n\tau)S_{\tau}((n-a) \tau)
 \Pi_{\tau/4} \nonumber\\
&=[J(n\tau),\Pi_\tau]S ((n-a) \tau)
 \Pi_{\tau/4} + \Pi_\tau J(n\tau)S ((n-a) \tau)
 \Pi_{\tau/4} \nonumber\\
&=[J(n\tau),\Pi_\tau]S_{\tau/4} ((n-a) \tau)
 + \Pi_\tau S ((n-a) \tau)J(a\tau)
 \Pi_{\tau/4} \nonumber\\
&=[J(n\tau),\Pi_\tau]S_{\tau/4} ((n-a) \tau)
 + S_\tau ((n-a) \tau)[J(a\tau), \Pi_{\tau/4} ]\nonumber\\
&\quad+  S_\tau ((n-a) \tau)J(a\tau).\label{eq:factorPi}
\end{align}
Now, $J$ acts on gauge invariant nonlinearities like the
gradient, see \eqref{eq:Jder}, and we observe that $(N(\tau)-\Id)Z_\tau$ enjoys
this gauge invariance. We infer from \eqref{eq:factorPi} and
Lemma~\ref{lem:cutoff} that the above estimates, proven in the case
$A\in \{\Id,\nabla\}$, remains essentially the same in the case
$A=J$:
\begin{align*}
&\quad \sum_{(q,r)\in \{(q_0,r_0),(\infty,2)\}} \|J Z_\tau \|_{\ell^{q}(a\le n\le
  a+N_*+1;L^{r})}\\
  &\le (1+C\tau^{1/2})\Big(2K_0 \|Z_\tau\|_{X(\{a\tau\})}\Big.\\
  &\qquad\qquad\qquad\quad\Big.+ K_1 \left\|
    \<n\tau\>^{-\delta(r_0)}\right\|_{\ell^\gamma(a\le n\le
  a+N_*)}^{2\si} \(4K_0\|Z_\tau\|_{X(\{a\tau\})}\)^{2\si+1}\Big)\\
  &\le 3K_0 \|Z_\tau\|_{X(\{a\tau\})}+ \frac{3}{2}K_1 \left\|
    \<n\tau\>^{-\delta(r_0)}\right\|_{\ell^\gamma(a\le n\le
  a+N_*)}^{2\si} \(4K_0\|Z_\tau\|_{X(\{a\tau\})}\)^{2\si+1},
\end{align*}
when $\tau\le \tau_0$.
The maximality of $N_*$ implies that
\begin{equation*}
  3K_1 \left\|
    \<n\tau\>^{-\delta(r_0)}\right\|_{\ell^\gamma(a\le n\le
  a+N_*)}^{2\si} \(4K_0\|Z_\tau\|_{X(\{a\tau\})}\)^{2\si+1}>2K_0 \|Z_\tau\|_{X(\{a\tau\})},
\end{equation*}
i.e.,
\[6K_1 \left\|
    \<n\tau\>^{-\delta(r_0)}\right\|_{\ell^\gamma(a\le n\le
  a+N_*)}^{2\si} >\(4K_0\|Z_\tau\|_{X(\{a\tau\})}\)^{-2\si},\]
since if the reverse inequality were true, then $N_*+1$ would belong to
$\Lambda$. Keeping this constant $K_1$ in the statement of
Proposition~\ref{prop:local-stab}, we see that
$I\subsetneqq [a\tau, (a+N_*)\tau]$, and the result follows, by using
Strichartz estimates once more
in order to cover all admissible pairs $(q,r)$.
\end{proof}

%%%%%%%%%%%%%%%%%%%%%%%%%%%%%%%%%%%%%%%%%%%%%%%%%%%%%%%%%%%%%%%%%%%%%%%%%%%%%%%%%

\section{More estimates}\label{sec:more}

The following lemma is the large time counterpart of
\cite[Lemma~5.2]{ChoiKoh2021}, where we also change the numerology:
\begin{lemma}\label{lem:more1}
For any time interval $I$, and $A\in \{\Id,\nabla,J\}$,
 \begin{align*}
    &\left\| A\(\frac{N(\tau)-\Id}{\tau} v - \frac{N(\tau)-\Id}{\tau}
      w \)\right\|_{\ell^{q_0'} (I; L^{r_0'})} \\
    &\quad\lesssim  \(\|v\|_{\ell^{\gamma} (I;L^{r_0})}^{2\si} +
\|w\|_{\ell^{\gamma} (I;L^{r_0})}^{2\si} \)\| A(v-w)\|_{\ell^{q_0}
      (I;L^{r_0})} \\
   &\qquad+ \( \|v\|_{\ell^\gamma(I;L^{r_0})}^{2\si-1}+
      \|w\|_{\ell^\gamma(I;L^{r_0})}^{2\si-1} \)
      \|v-w\|_{\ell^\gamma(I;L^{r_0})}\|A w\|_{\ell^{q_0}(I;
     L^{r_0})} \\
   &\qquad + \tau \|v-w\|_{\ell^\gamma(I;L^{r_0})}
    \|A w\|_{\ell^{q_0}(I; L^{r_0})} \left\|  |v|^{4\si-1}+
     |w|^{4\si-1}\right\|_{\ell^\frac{\gamma}{2\si-1}(I;L^{\frac{r_0}{2\si-1}})}.
    \end{align*}
 \end{lemma}

\begin{proof}
In view of Lemma~\ref{lem:propNL}, H\"older inequality yields
\begin{align*}
\Big\| \frac{N(\tau)-\Id}{\tau} v - \frac{N(\tau)-\Id}{\tau} w
  \Big\|_{\ell^{q_0'} (I; L^{r_0'})}    &\lesssim \big(\|v\|_{\ell^{\gamma} (I;L^{r_0})}^{2\si} +
\|w\|_{\ell^{\gamma} (I;L^{r_0})}^{2\si} \big)\| v-w\|_{\ell^{q_0} (I;L^{r_0})} .
 \end{align*}
Next, by differentiating and rearranging, we have
    \begin{align*}
    &\quad\nabla \left( \frac{N(\tau)-\Id}{\tau} v\right) -\nabla \left( \frac{N(\tau)-\Id}{\tau} w\right) \\
    &= \nabla \left( \frac{e^{-i\tau |v|^{2\si}}-1}{\tau} v\right) -\nabla \left( \frac{e^{-i\tau  |w|^{2\si}}-1}{\tau} w\right) \\
    &= -i\si \(e^{-i\tau |v|^{2\si}}|v|^{2\si} \nabla v
      -e^{-i\tau  |w|^{2\si}}  |w|^{2\si} \nabla w  \) \\
   &\quad -i\si \(e^{-i\tau |v|^{2\si}}|v|^{2\si-2}v^2 \nabla
     \bar v
      -e^{-i\tau  |w|^{2\si} } |w|^{2\si-2} w^2\nabla \bar w  \) \\
      &\quad+ \(\( \frac{e^{-i\tau |v|^{2\si}}-1}{\tau} \) -
        \( \frac{e^{-i\tau |w|^{2\si}}-1}{\tau} \) \)\nabla w +
        \( \frac{e^{-i\tau |v|^{2\si}}-1}{\tau} \)( \nabla v - \nabla w).
    \end{align*}
Lemma~\ref{lem:propNL} now yields
    \begin{equation}\label{eq:1402}
    \begin{aligned}
    &\left\| \nabla \( \frac{N(\tau)-\Id}{\tau} v\) -\nabla \( \frac{N(\tau)-\Id}{\tau} w\) \right\|_{\ell^{q_0'}(I; L^{r_0'})} \\
    &\quad \le \si \left\| e^{-i\tau |v|^{2\si}}|v|^{2\si} \nabla v
      -e^{-i\tau  |w|^{2\si}}  |w|^{2\si} \nabla w\right\|_{\ell^{q_0'}(I; L^{r_0'})} \\
   &\qquad + \si \left\| e^{-i\tau |v|^{2\si}}|v|^{2\si-2}v^2 \nabla
     \bar v
      -e^{-i\tau  |w|^{2\si} } |w|^{2\si-2} w^2\nabla \bar w  \right\|_{\ell^{q_0'}(I; L^{r_0'})} \\
    &\qquad + \left\| \left||v|^{2\si}-|w|^{2\si}\right|\nabla w
    \right\|_{\ell^{q_0'}(I; L^{r_0'})} + \left\| |v|^{2\si}|\nabla
      v-\nabla w|    \right\|_{\ell^{q_0'}(I; L^{r_0'})} .
    \end{aligned}
    \end{equation}
By inserting the terms $|w|^{2\si}\nabla w e^{-i\tau |v|^{2\si}}$ and
$|w|^{2\si-2} w^2\nabla \bar w e^{-i\tau |v|^{2\si}}$ in the first and
second lines of the right hand side of \eqref{eq:1402}, respectively,
we obtain, since $\si\ge 1/2$,
    \begin{equation}\label{eq:1503}
      \begin{aligned}
  &\left\| \nabla \( \frac{N(\tau)-\Id}{\tau} v\) -\nabla \( \frac{N(\tau)-\Id}{\tau} w\) \right\|_{\ell^{q_0'}(I; L^{r_0'})} \\
  &\quad \lesssim     \left\| |v|^{2\si}(\nabla v-\nabla
    w)\right\|_{\ell^{q_0'}(I; L^{r_0'})}   \\
  & \quad\quad+
\left\| \(|v|^{2\si-1}+ |w|^{2\si-1}\)\(|v|-|w|\)\nabla  w\right\|_{\ell^{q_0'}(I;
  L^{r_0'})} \\
&\quad\quad+ \left\| |w|^{2\si}\nabla
  w\(e^{-i\tau |v|^{2\si}}- e^{-i\tau  |w|^{2\si}}  \)\right\|_{\ell^{q_0'}(I; L^{r_0'})}.
    \end{aligned}
  \end{equation}
  The first term on the right hand side is estimated by H\"older
  inequality \eqref{eq:holder},
  \begin{align*}
    \left\| |v|^{2\si}(\nabla v-\nabla
    w)\right\|_{\ell^{q_0'}(I; L^{r_0'})} \le \|v\|_{\ell^\gamma(I;L^{r_0})}^{2\si}
    \|\nabla v-\nabla    w\|_{\ell^{q_0}(I; L^{r_0})} .
  \end{align*}
  Similarly,
  \begin{align*}
    &\left\| \(|v|^{2\si-1}+ |w|^{2\si-1}\)\(|v|-|w|\)\nabla  w\right\|_{\ell^{q_0'}(I;
      L^{r_0'})} \\
    &\quad\lesssim \( \|v\|_{\ell^\gamma(I;L^{r_0})}^{2\si-1}+
      \|w\|_{\ell^\gamma(I;L^{r_0})}^{2\si-1} \)
      \|v-w\|_{\ell^\gamma(I;L^{r_0})}\|\nabla w\|_{\ell^{q_0}(I; L^{r_0})} .
  \end{align*}
For the last term in \eqref{eq:1503}, we notice the pointwise
estimate, using again $\si\ge 1/2$,
\begin{equation*}
    \left||w|^{2\si}\nabla
  w\(e^{-i\tau |v|^{2\si}}- e^{-i\tau  |w|^{2\si}}  \)\right|
    \lesssim \tau|w|^{2\si} |\nabla w|\(|v|^{2\si-1} + |w|^{2\si-1} \) |v-w|.
\end{equation*}
H\"older inequality now yields
\begin{align*}
 & \left\| |w|^{2\si}\nabla  w\(e^{-i\tau |v|^{2\si}}-
  e^{-i\tau  |w|^{2\si}}  \)\right\|_{\ell^{q_0'}(I; L^{r_0'})}\\
  &\quad\lesssim \tau \|v-w\|_{\ell^\gamma(I;L^{r_0})}
    \|\nabla w\|_{\ell^{q_0}(I; L^{r_0})}
    \left\| |v|^{4\si-1}+|w|^{4\si-1}\right\|_{\ell^\frac{\gamma}{2\si-1}(I;L^{\frac{r_0}{2\si-1}})}.
\end{align*}
The lemma follows from the above estimates, for $A\in
\{\Id,\nabla\}$. The case $A=J$ is similar to the case $A=\nabla$, since
$J(t)=S(t)xS(-t)$ and $J$ acts on gauge invariant nonlinearities like
the gradient, \eqref{eq:Jder}.
\end{proof}

So far, we have supposed $\si\ge 1/2$. From the next proposition on,
we need to require $\si>1/2$, as pointed out in the proof below.

\begin{proposition}\label{prop:discrete-stability}
Suppose $\si>1/2$. Let $(q,r)$ be an admissible pair. Consider an
  interval $I$ of the form $I=\overline{[a\tau,b\tau)}$, with $a,b\in \N\cup\{\infty\}$ such
  that $0\le a<b\le \infty$. There exists $K_2$ such that if
\begin{equation}\label{eq:K2}
 M   \left\|\<n\tau\>^{-\delta(r_0)}\right\|_{\ell^\gamma(I)}\le
     K_2,
   \end{equation}
   then the following holds.
  \begin{itemize}
    \item There exists $C=C(d,\si,q)$ such that if $\phi_1,\phi_2\in \Sigma$ with
      $\|Z_\tau^{\phi_j}\|_{X(\{a\tau\})}\le M$, $j=1,2$, then
  \begin{equation}
    \label{eq:lipschitz}
 Q_1=\max_{A\in \{\Id,\nabla,J\}} \|A(n\tau)\(Z^{\phi_2}_\tau(n\tau) -
 Z^{\phi_1}_\tau(n\tau)
  \)\|_{\ell^q(I;L^r)}\le C
 \|Z_\tau^{\phi_2}-Z_\tau^{\phi_1}\|_{X(\{a\tau\})},
  \end{equation}
  where $Z_\tau^{\phi_j}(n\tau)=(S_\tau(\tau)N(\tau))^n\Pi_\tau \phi_j$.
  \medskip

\item If $\psi\in \Sigma\cap H^2(\R^d)$, with $\|Z_\tau^{\psi}\|_{X(\{a\tau\})}\le
  M$, then there exists a constant $C=C(d,\si,M,\psi)>0$ such that
  \begin{equation}
    \label{eq:local-error-H2}
   Q_2= \max_{A\in \{\Id,\nabla,J\}} \left\|A(n\tau) \(Z^\psi_\tau(n\tau)
-u^\psi(n\tau)\)\right\|_{\ell^q(I;L^r)}\le
     C \tau^{1/2}.
  \end{equation}
\item Assume $\phi\in \Sigma$ satisfies $\|\phi\|_{X(\{a\tau\})}\le
  M/2$. denote $u^\phi$ by the solution of \eqref{eq:NLS} with the initial condition $u_{\mid
  t=a\tau}=\phi$, i.e., $u$ is given by \eqref{eq:duhamel} with
$t_0=a\tau$. Similarly in the discrete version, let
$Z^\phi_\tau(n\tau)$ $(n>a)$ be given  by \eqref{eq:discrete-duhamel} with  $Z_\tau(a\tau)$ replaced by $\phi$. Then it holds that
  \begin{equation}
    \label{eq:loc-stab-Sigma}
    \lim_{\tau\to 0}\max_{A\in \{\Id,\nabla,J\}}
    \left\|A(n\tau)\(Z^\phi_\tau(n\tau)-u^\phi(n\tau)\) \right\|_{\ell^\infty(I;L^2)}=0.
  \end{equation}
\end{itemize}
\end{proposition}
\begin{remark}\label{rem:series2}
  In view of Remark~\ref{rem:series}, the assumption \eqref{eq:K2}
  means, typically for $M$ large, that either $I$ is finite with $|I|$
  sufficiently small, say $|I|=T\approx M^{-\gamma}$, or
  $I=[NT,\infty)$, with
  \begin{equation*}
    N \approx \frac{1}{T}M^{\frac{\gamma}{\gamma\delta(r_0)-1}}\approx
    M^{\gamma+\frac{\gamma}{\gamma\delta(r_0)-1}}.
  \end{equation*}
\end{remark}

\begin{proof}
 We divide the proof into three  parts, corresponding to
 \eqref{eq:lipschitz}, \eqref{eq:local-error-H2} and
 \eqref{eq:loc-stab-Sigma}, respectively. \\

\noindent \emph{Proof of \eqref{eq:lipschitz}.}
Let $\phi_1$ and $\phi_2 \in \Sigma$ such that $\|Z_\tau^{\phi_j}\|_
{X(\{a\tau\})}\le M$, $j=1,2$. We consider the difference between the
Duhamel formulas of $Z_{\tau}^{\phi_1}$ and $Z_{\tau}^{\phi_2}$
provided by \eqref{eq:discrete-duhamel}. For $A\in \{\Id,\nabla,  J\}$, \eqref{eq:discrete-duhamel} together with Corollary \ref{cor:stri-disc} yields
\begin{align*}
  &\quad \|A(n\tau)\(Z^{\phi_2}_\tau(n\tau) - Z^{\phi_1}_\tau(n\tau)
  \)\|_{\ell^q(I;L^r)}\\
  &\lesssim\left\|Z_\tau^{\phi_2}(a\tau)-Z_\tau^{\phi_1}(a\tau)\right\|_{L^2}+                           \left\|A(a\tau)\(Z_\tau^{\phi_2}(a\tau)-Z_\tau^{\phi_1}(a\tau)\)\right\|_{L^2} \\
  &\quad+\left\|\frac{N(\tau)-\Id}{\tau}Z^{\phi_2}_\tau -
  \frac{N(\tau)-\Id}{\tau}Z^{\phi_1}_\tau\right\|_{\ell^{q_0'}(I;L^{r_0'})}\\
  &\quad+
  \left\|A\(\frac{N(\tau)-\Id}{\tau}Z^{\phi_2}_\tau -
  \frac{N(\tau)-\Id}{\tau}Z^{\phi_1}_\tau\)\right\|_{\ell^{q_0'}(I;L^{r_0'})}.
\end{align*}
In view of Proposition~\ref{prop:local-stab} with choosing $K_2 >0$
small enough in \eqref{eq:K2}, for $(q,r) \in \{(q_0, r_0), (\infty,2)\}$, we have
    \begin{equation}\label{eq:borne-unif-phij}
    \|A(n\tau) Z_{\tau}^{\phi_j}(n \tau)\|_{\ell^q (I; L^{r} )}
    \le K M, \quad j=1,2,
  \end{equation}
  for some constant $K>0$. In view of this estimate,
  Lemma~\ref{lem:more1} implies, along with \eqref{eq:GNlibre} to
  estimate $\|\phi_2-\phi_1\|_{\ell^\gamma(I;L^{r_0})}$,
\begin{align*}
 Q_1&=\max_{A\in \{\Id,\nabla,J\}}  \|A(n\tau)\(Z^{\phi_2}_\tau(n\tau) - Z^{\phi_1}_\tau(n\tau)
  \)\|_{\ell^q(I;L^r)}\\
  &\lesssim \max_{A\in \{\Id,\nabla,J\}}\left\|A(a\tau)\(Z_\tau^{\phi_2}(a\tau)-Z_\tau^{\phi_1}(a\tau)\)\right\|_{L^2}\\
  &\quad+\max_{A\in \{\Id,\nabla,J\}}
 \left\|A\(\frac{N(\tau)-\Id}{\tau}Z^{\phi_2}_\tau -
  \frac{N(\tau)-\Id}{\tau}Z^{\phi_1}_\tau\)\right\|_{\ell^{q_0'}(I;L^{r_0'})}\\
    &\lesssim\max_{A\in \{\Id,\nabla,J\}}\left\|A(a\tau)\(Z_\tau^{\phi_2}(a\tau)-Z_\tau^{\phi_1}(a\tau)\)\right\|_{L^2} \\
  &+\(M \left\|\<n\tau\>^{-\delta(r_0)}\right\|_{\ell^\gamma(I)}\)^{2\si}\max_{A\in \{\Id,\nabla,J\}}\|A(Z^{\phi_2}_\tau(n\tau)
    - Z^{\phi_1}_\tau(n\tau))\|_{\ell^{q_0}(I;L^{r_0})}\\
  & +\(M
    \left\|\<n\tau\>^{-\delta(r_0)}\right\|_{\ell^\gamma(I)}\)^{2\si}
    \|Z^{\phi_2}_\tau(n\tau) - Z^{\phi_1}_\tau(n\tau)\|_{\Gamma(I)}\\
  & + \tau M
    \left\|\<n\tau\>^{-\delta(r_0)}\right\|_{\ell^\gamma(I)}
    \|Z^{\phi_2}_\tau(n\tau) -
    Z^{\phi_1}_\tau(n\tau)\|_{\Gamma(I)}\times\\
  &\qquad\times\(
    \|Z^{\phi_2}_\tau(n\tau)\|_{\ell^{\frac{4\si-1}{2\si-1}\gamma}(I;L^{\frac{4\si-1}{2\si-1}r_0})}^{4\si-1} +\|Z^{\phi_1}_\tau(n\tau)\|_{\ell^{\frac{4\si-1}{2\si-1}\gamma}(I;L^{\frac{4\si-1}{2\si-1}r_0})}^{4\si-1} \)
    ,
\end{align*}
where we recall that the norm $\|\cdot\|_{\Gamma(I)}$ is defined in
\eqref{eq:chi}.
In view of Bernstein inequality \eqref{eq-2-22}, for $j=1,2$,
\begin{equation*}
  \|Z^{\phi_j}_\tau(n\tau)\|_{L^{\frac{4\si-1}{2\si-1}r_0}}\lesssim
  \tau^{\frac{d}{2}\( \frac{2\si-1}{(4\si-1)r_0}- \frac{1}{r_0}\)}
  \|Z^{\phi_j}_\tau(n\tau)\|_{L^{r_0}},
\end{equation*}
and
\begin{align*}
  \frac{d}{2}\( \frac{2\si-1}{(4\si-1)r_0}- \frac{1}{r_0}\)&=
 -\frac{d}{2}\frac{2\si}{(4\si-1)r_0}= -\frac{d\si}{(4\si-1)(2\si+2)}=
  -\frac{2}{(4\si-1)q_0},
\end{align*}
so we obtain
\begin{equation}
  \label{eq:4sigma-1}
  \begin{aligned}
 \|Z^{\phi_j}_\tau(n\tau)\|_{\ell^{\frac{4\si-1}{2\si-1}\gamma}(I;L^{\frac{4\si-1}{2\si-1}r_0})}^{4\si-1}
  &\lesssim \tau^{-2/q_0}
    \|Z^{\phi_j}_\tau(n\tau)\|_{\ell^{\frac{4\si-1}{2\si-1}\gamma}(I;L^{r_0})}^{4\si-1}\\
  &\lesssim \tau^{-2/q_0}\|\<n\tau\>^{-\delta(r_0)}
    \|^{4\si-1}_{\ell^{\frac{4\si-1}{2\si-1}\gamma}(I)}
    \|Z^{\phi_j}_\tau\|_{\Gamma(I)}^{4\si-1}\\
  & \lesssim \tau^{-2/q_0}
    \|\<n\tau\>^{-\delta(r_0)} \|^{4\si-1}_{\ell^{\gamma}(I)}
    M^{4\si-1},
\end{aligned}
\end{equation}
where we have used \eqref{eq:borne-unif-phij} and
$\frac{4\si-1}{2\si-1}>1$ for the last
inequality. Note that if $\si=1/2$, the Lebesgue
  exponent $\frac{4\si-1}{2\si-1}\gamma$ becomes infinite, and we can
  no longer invoke \eqref{eq-2-22} like we did above. This is why we
  assume $\si>1/2$. In view of \eqref{eq:gamma}, we
get
\begin{align*}
  Q_1& \lesssim\max_{A\in \{\Id,\nabla,J\}}
   \left\|A(a\tau)\(Z_\tau^{\phi_2}(a\tau)-Z_\tau^{\phi_1}(a\tau)\)\right\|_{L^2} \\
  &\quad +\(M
    \left\|\<n\tau\>^{-\delta(r_0)}\right\|_{\ell^\gamma(I)}\)^{2\si}
    \max_{A\in \{\Id,\nabla,J\}}\|A(n\tau)(Z^{\phi_2}_\tau(n\tau)
    - Z^{\phi_1}_\tau(n\tau))\|_{\ell^{q_0}(I;L^{r_0})}\\
  &\quad +\(M
    \left\|\<n\tau\>^{-\delta(r_0)}\right\|_{\ell^\gamma(I)}\)^{2\si}
    \|Z^{\phi_2}_\tau(n\tau) - Z^{\phi_1}_\tau(n\tau)\|_{\Gamma(I)}\\
  &\quad + \( M
    \left\|\<n\tau\>^{-\delta(r_0)}\right\|_{\ell^\gamma(I)}\)^{4\si}
    \|Z^{\phi_2}_\tau(n\tau) -
    Z^{\phi_1}_\tau(n\tau)\|_{\Gamma(I)}\\
    &\le C_1\left\|Z_\tau^{\phi_2}(a\tau)-Z_\tau^{\phi_1}(a\tau)
    \right\|_{X(\{a\tau\})}\\
    &\quad +C_2 Q_1\left[\( M
    \left\|\<n\tau\>^{-\delta(r_0)}\right\|_{\ell^\gamma(I)}\)^{2\si}+
    \( M    \left\|\<n\tau\>^{-\delta(r_0)}\right\|_{\ell^\gamma(I)}\)^{4\si}\right].
\end{align*}
The inequality \eqref{eq:lipschitz} then follows in the same fashion
as in the proof of Proposition~\ref{prop:local-stab}, up to
considering a smaller constant $K_2$.
\bigbreak

\noindent\emph{Proof of \eqref{eq:local-error-H2}.}
Assume that $\psi \in \Sigma\cap H^2 (\R^d)$, with
$\|Z_\tau^\psi\|_{X(\{a\tau\})}\le M$, and let $A\in \{\Id,\nabla,J\}$.
We estimate $Z_{\tau}^{\psi}(n \tau) -
\Pi_{\tau} u^{\psi}(n \tau)$ instead of $Z_{\tau}^{\psi}(n \tau) -
u^{\psi}(n \tau)$, since we have
    \begin{align*}
    \left\| A(n\tau)\(u^{\psi}(n \tau) - \Pi_{\tau} u^{\psi}(n \tau)
      \)\right\|_{\ell^{q} (I; L^{r})}
    &\lesssim \tau^{1/2} \| A (n\tau)u^{\psi} (n \tau)\|_{\ell^{q} (I; W^{1,r})} \lesssim \tau^{1/2} ,
    \end{align*}
thanks to \eqref{eq-2-20}, \eqref{eq:cutoffJ} and
Theorem~\ref{theo:B}.
We now decompose $Z_{\tau}^{\psi}(n\tau) - \Pi_{\tau}u^{\psi} (n\tau)$ in
view of  Duhamel formulas \eqref{eq:discrete-duhamel} and
\eqref{eq:duhamel} (recalling that $S_\tau(t)= S(t)\Pi_\tau= \Pi_\tau S(t)$), as
    \begin{equation*}
    \begin{aligned}
    &\quad Z_{\tau}^{\psi}(n \tau) - \Pi_{\tau}u^{\psi} (n \tau) \\
    &=S_\tau((n-a)\tau)(Z_\tau(a\tau)-u^\psi(a\tau))\\
    &\quad+ \underbrace{\tau \sum_{k=a}^{n-1} S_{\tau}(n \tau -k \tau) \left( \frac{N(\tau) -\Id}{\tau} Z_{\tau}^{\psi}(k\tau)
        - \frac{N(\tau) -\Id}{\tau} \Pi_{\tau} u^{\psi}(k\tau)
      \right)}_{=: \mathcal G_1} \\
    &\qquad +\underbrace{\tau \sum_{k=a}^{n-1} S_{\tau}(n \tau -k \tau) \frac{N(\tau) -\Id}{\tau} \Pi_{\tau} u^{\psi}(k\tau)
        + i\int_{a\tau}^{n\tau} S_{\tau}(n\tau -s) |u^{\psi}|^{2\si}
        u^{\psi} (s) ds}_{=:\mathcal G_2}.
    \end{aligned}
  \end{equation*}
  In this decomposition, $\mathcal G_1$ must be thought of as an
  arbitrary small perturbation of the left hand side, and  $\mathcal
  G_2$ must be thought of as the actual source term.
  \smallbreak

Let $(q,r)$ be an admissible pair, and $A\in \{\Id,\nabla\}$. We estimate the terms of the above right hand side, after the action of
$A$. First, discrete Strichartz estimates \eqref{eq:disc-stri-inhom}
and Lemma~\ref{lem:more1} yield
\begin{align*}
  \|A\mathcal G_1\|_{\ell^qL^r}&\lesssim \left\|A\(\frac{N(\tau)
  -\Id}{\tau} Z_{\tau}^{\psi}(k\tau)  - \frac{N(\tau)
 -\Id}{\tau} \Pi_{\tau}  u^{\psi}(k\tau)\)\right\|_{\ell^{q_0'}L^{r_0'}}\\
  &\hspace{-6mm}\lesssim  \(\|Z^\psi_\tau\|_{\ell^{\gamma} L^{r_0}}^{2\si} +
\|\Pi_{\tau}u^\psi\|_{\ell^{\gamma}L^{r_0}}^{2\si} \)\| A(Z^\psi_\tau-\Pi_{\tau}u^\psi)\|_{\ell^{q_0}L^{r_0}} \\
   &\hspace{-6mm}+ \( \|Z^\psi_\tau\|_{\ell^\gamma L^{r_0}}^{2\si-1}+
      \|\Pi_{\tau}u^\psi\|_{\ell^\gamma L^{r_0}}^{2\si-1} \)
      \|Z^\psi_\tau-\Pi_{\tau}u^\psi\|_{\ell^\gamma L^{r_0}}\|A \Pi_{\tau}u^\psi\|_{\ell^{q_0} L^{r_0}} \\
   &\hspace{-6mm}+ \tau \|Z^\psi_\tau-\Pi_{\tau}u^\psi\|_{\ell^\gamma L^{r_0}}
     \|A \Pi_{\tau}u^\psi\|_{\ell^{q_0} L^{r_0}} \left\|  |Z^\psi_\tau|^{4\si-1}+
     |\Pi_{\tau}u^\psi|^{4\si-1}\right\|_{\ell^\frac{\gamma}{2\si-1}L^{\frac{r_0}{2\si-1}}},
\end{align*}
where the interval $I$ is omitted to ease notations.
Recalling Proposition~\ref{prop:local-stab} and
\eqref{eq:apriori-univ}, we infer
\begin{equation*}
   \max_{B\in \{\Id,\nabla,J\}} \sum_{(q,r)\in
      \{(q_0,r_0),(\infty,2)\}}\| B Z^\psi_{\tau}\|_{\ell^q(I;L^r)}\lesssim M.
\end{equation*}
Invoking \eqref{eq:4sigma-1},
\begin{equation*}
  \||Z^{\psi}_\tau|^{4\si-1}\|_{\ell^{\frac{\gamma}{2\si-1}}(I;L^{\frac{r_0}{2\si-1}})}
  \lesssim \tau^{-\frac{2}{q_0}}
    \|\<n\tau\>^{-\delta(r_0)} \|^{4\si-1}_{\ell^{\gamma}(I)}
    M^{4\si-1}.
\end{equation*}
In view of Theorem~\ref{theo:B}, \eqref{eq:cutoffJ} and \eqref{eq-2-20'}, we also have
\begin{equation*}
   \max_{B\in \{\Id,\nabla,J\}} \sum_{(q,r)\in
      \{(q_0,r_0),(\infty,2)\}}\| B \Pi_{\tau}u^\psi\|_{\ell^q(I;L^r)}\lesssim M,
\end{equation*}
and, like for \eqref{eq:4sigma-1},
\begin{equation*}
   \||\Pi_\tau u^\psi|^{4\si-1}\|_{\ell^{\frac{\gamma}{2\si-1}}(I;L^{\frac{r_0}{2\si-1}})}
  \lesssim \tau^{-\frac{2}{q_0}}
    \|\<n\tau\>^{-\delta(r_0)} \|^{4\si-1}_{\ell^{\gamma}(I)}
    M^{4\si-1}.
  \end{equation*}
  We deduce the (uniform in $\tau\in (0,1)$) estimate
\begin{align*}
  \|A\mathcal G_1\|_{\ell^q(I;L^r)}
  &\lesssim
    \( \left\|\<n\tau\>^{-\delta(r_0)}\right\|_{\ell^\gamma(I)}M\)^{2\si}
    \| A(Z^\psi_\tau-\Pi_{\tau}u^\psi)\|_{\ell^{q_0}(I;L^{r_0})} \\
  &\quad + \( \left\|\<n\tau\>^{-\delta(r_0)}\right\|_{\ell^\gamma(I)}M\)^{2\si-1}M
    \|Z^\psi_\tau-\Pi_{\tau}u^\psi\|_{\ell^\gamma (I;L^{r_0})}\\
  &\quad + \( \left\|\<n\tau\>^{-\delta(r_0)}\right\|_{\ell^\gamma(I)}M\)^{4\si-1}M
    \|Z^\psi_\tau-\Pi_{\tau}u^\psi\|_{\ell^\gamma(I; L^{r_0})}.
  \end{align*}
  As before, we can estimate
  \[
  \|Z^\psi_\tau-\Pi_{\tau}u^\psi\|_{\ell^\gamma (I;L^{r_0})}\lesssim
  \left\|\<n\tau\>^{-\delta(r_0)}\right\|_{\ell^\gamma(I)}\|Z^\psi_\tau-\Pi_{\tau}u^\psi\|_{\Gamma(I)}\\
\lesssim Q_2 \left\|\<n\tau\>^{-\delta(r_0)}\right\|_{\ell^\gamma(I)}.\]
When $A=J$, we recall that in view of Corollary~\ref{cor:stri-disc},
we may repeat the same computations, up to an (irrelevant) extra multiplicative
factor $1+C\tau^{1/2}$.
Picking $\|\<n\tau\>^{-\delta(r_0)}
 \|_{\ell^{\gamma}(I)}M$ sufficiently small,
 \begin{equation*}
 \max_{{A \in\{\Id,\nabla,J\}}\atop{(q,r) \,\, \mathrm{admissible}}} \|A\mathcal G_1\|_{\ell^q(I;L^r)}
   \le \frac{1}{2}Q_2 ,
\end{equation*}
and thus
\begin{equation*}
  \max_{\atop{(q,r)\,\,\mathrm{admissible}}} Q_2\le C\|Z_\tau(a\tau)-u^\psi(a\tau)\|_{L^2}+2\max_{{A \in\{\Id,\nabla,J\}}\atop{(q,r)\,\,\mathrm{admissible}
      }} \|A\mathcal G_2 \|_{\ell^q(I;L^r)} .
\end{equation*}
By adapting of the proof of Theorem~\ref{theo:B}, we see
that
\begin{equation*}
  \max_{{A \in\{\Id,\nabla,J\}}\atop{(q,r)\,\,\mathrm{admissible}}} \|A\mathcal G_2
  \|_{\ell^q(I;L^r)}\le C(d,\si,M, \psi) \tau^{1/2},
\end{equation*}
hence \eqref{eq:local-error-H2} is established by similar induction applied in the proof of Lemma~\ref{lem:A4}.
\bigbreak

\noindent\emph{Proof of \eqref{eq:loc-stab-Sigma}.}
Let $\eps\in (0,M/2)$, and $\psi\in H^2\cap \Sigma$ such that
\begin{equation*}
 \|\phi-\psi\|_{X(\{a\tau\})}
 \le \eps.
\end{equation*}
Then $\|\psi\|_{X(\{a\tau\})}\le M$, so we may invoke
Theorem~\ref{theo:B} and \eqref{eq:lipschitz} to claim that
\begin{equation*}
  \|u^\phi-u^\psi\|_{Y(\R)}\le C\|\phi-\psi\|_{X(\{a\tau\})}\lesssim \eps,
\end{equation*}
and, since $u^\phi(a\tau)=Z^\phi_\tau(a\tau)=\phi$,
\begin{equation*}
  \max_{A \in\{\Id,\nabla,J\}}\|A(n\tau)\(Z^\phi_\tau(n\tau)-
  Z^\psi_\tau(n\tau)\)\|_{\ell^\infty (I;L^2)}\le
  C\|\phi-\psi\|_{X(\{a\tau\})}\lesssim \eps.
\end{equation*}
On the other hand, \eqref{eq:local-error-H2} yields
\begin{equation*}
  \max_{A \in\{\Id,\nabla,J\}}\|A(n\tau)\(Z^\psi_\tau(n\tau)-
  u^\psi(n\tau)\)\|_{\ell^\infty (I;L^2)}\le C(d,\si,M,\psi)\tau^{1/2},
\end{equation*}
and the triangle inequality implies
\begin{equation*}
   \max_{A \in\{\Id,\nabla,J\}}\|A(n\tau)\(Z^\phi_\tau(n\tau)-
  u^\phi(n\tau)\)\|_{\ell^\infty (I;L^2)}\le
  C(d,\si,M,\psi)\tau^{1/2}+ C(d,\si,M)\eps.
\end{equation*}
Therefore, for all $\eps\in (0,M/2)$,
\begin{equation*}
  \limsup_{\tau\to 0} \max_{A \in\{\Id,\nabla,J\}}\|A(n\tau)\(Z^\phi_\tau(n\tau)-
  u^\phi(n\tau)\)\|_{\ell^\infty (I;L^2)}\le C(d,\si,M)\eps,
\end{equation*}
and the left hand side must then be zero.
\end{proof}

%%%%%%%%%%%%%%%%%%%%%%%%%%%%%%%%%%%%%%%%%%%%%%%%%%%%%%%%%%%%%%%%%%%%%%%%%%%%%%%%%%%%%%%%%%%%%

\section{Global $\Sigma$ stability of $Z_{\tau}$}\label{sec:stability}

In this section, we prove Theorem~\ref{theo:stability}. As suggested
by the statements of Propositions~\ref{prop:local-stab} and
\ref{prop:discrete-stability}, the idea is to split $\N$ into finitely
many intervals, like in the proof of Theorem~\ref{theo:CVL2}, and
apply these local results on each of them: the
accumulation of errors is thus limited.

\begin{proof}[Proof of Theorem \ref{theo:stability}]
  Let $\phi \in \Sigma$: in view of Theorem~\ref{theo:aprioriX}, there
  exists $M$ such that
  \begin{equation*}
    \sum_{A\in \{\Id,\nabla,J\}} \|Au^\phi\|_{L^\infty(\R;L^2)}
   % +\|Au^\phi\|_{L^q(\R;L^r)}  \)
\le \frac{M}{2}.
  \end{equation*}
 Consider $K_2$ provided by Proposition~\ref{prop:discrete-stability},
 see \eqref{eq:K2}. Like in the proof of
  Theorem~\ref{theo:CVL2}, we can find a finite number
$K$ of time intervals $I_j= \overline{[m_j\tau,m_{j+1}\tau)}$
with $0=m_1<m_2<\ldots<m_K\in \N$, $m_{K+1}=\infty$, such that
\begin{equation*}
   M\left\|\<n\tau\>^{-\delta(r_0)}\right\|_{L^\gamma(I_j)}
   \le K_2,\quad \R_+ = \bigcup_{j=1}^K I_j.
 \end{equation*}
  Typically, in view of Remarks~\ref{rem:series} and
  \ref{rem:series2}, this means that for $1\le j\le K-1$, we may
  consider $|I_j|= T$ sufficiently small (in terms of $M$, and
  uniformly in $\tau$), while $I_{K}$ is of the form $[\alpha T,\infty)$,
  with $\alpha$ sufficiently large.
  \smallbreak

For each $1\le j\le K$, let $\psi_j \in \Sigma\cap H^2 $ such that
    \begin{equation}\label{eq:psij}
    \| \psi_j - u^{\phi}(m_j\tau) \|_{X(\{m_j\tau\})} \le \frac{M}{(10 C_0)^{K}} ,
  \end{equation}
  with $C_0$ the largest constant between $1$, the constant $C$ in the first
  point of Theorem~\ref{theo:B}, and the constant $C$ in
  \eqref{eq:lipschitz}. We show by induction that for $\tau>0$
  sufficiently small,
  \begin{equation}\label{eq:induction1}
    \max_{A\in
      \{\Id,\nabla,J\}}\|A(n\tau)\(Z^\phi_\tau(n\tau)-
    u^\phi(n\tau)\)\|_{\ell^\infty(I_j;L^2)}\le
    (3C_0)^j  \frac{M}{(10 C_0)^{K}} ,
  \end{equation}
  and
  \begin{equation}\label{eq:induction2}
     \max_{A\in
      \{\Id,\nabla,J\}}\|A(n\tau)Z^\phi_\tau(n\tau)\|_{\ell^\infty(I_j;L^2)} \le M.
  \end{equation}
Let $A\in \{\Id,\nabla,J\}$. For $j=1$, we use triangle inequality to observe
\begin{align*}
  \|A\(Z^\phi_\tau(n\tau)-u^\phi(n\tau)\)\|_{\ell^\infty(I_1;L^2)}&\le
\|A\(Z^\phi_\tau(n\tau)-Z^{\psi_1}_\tau(n\tau)\)\|_{\ell^\infty(I_1;L^2)}\\
 &\quad +
\|A\(Z^{\psi_1}_\tau(n\tau)-u^{\psi_1}(n\tau)\)\|_{\ell^\infty(I_1;L^2)}
  \\
  &\quad+
  \|A\(u^{\psi_1}(n\tau)-u^\phi(n\tau)\)\|_{\ell^\infty(I_1;L^2)}.
\end{align*}
In view of \eqref{eq:lipschitz}, the first term in the right hand side
is controlled by
\begin{equation*}
  \|A\(Z^\phi_\tau(n\tau)-Z^{\psi_1}_\tau(n\tau)\)\|_{\ell^\infty(I_1;L^2)}
  \le C_0 \|\phi-\psi_1\|_{X(\{0\})}\le C_0\frac{M}{(10C_0)^{K}} ,
\end{equation*}
where we have used \eqref{eq:psij}. The second term is estimated in
view of \eqref{eq:local-error-H2}, by
\begin{equation*}
  \|A\(Z^{\psi_1}_\tau(n\tau)-u^{\psi_1}(n\tau)\)\|_{\ell^\infty(I_1;L^2)}\le
  \tau^{1/2}C(d,\si,M,\psi_1).
\end{equation*}
For $0<\tau\le \tau_1=\tau_1 (d,\si,M,\psi_1)$, we infer
\begin{equation*}
  \|A\(Z^{\psi_1}_\tau(n\tau)-u^{\psi_1}(n\tau)\)\|_{\ell^\infty(I_1;L^2)}\le
  C_0\frac{M}{(10 C_0)^{K}} .
\end{equation*}
We impose a smallness constraint on each $I_j$, but since there are
finitely many such intervals, the minimum of these $\tau_j$ is indeed
positive.
Finally, Theorem~\ref{theo:B} and \eqref{eq:psij} yield
\begin{equation*}
   \|A\(u^{\psi_1}(n\tau)-u^\phi(n\tau)\)\|_{\ell^\infty(I_1;L^2)}\le
   C_0\|\psi_1-\phi\|_{X(\{0\})}\le C_0\frac{M}{(10 C_0)^{K}} ,
\end{equation*}
hence \eqref{eq:induction1} for $j=1$. To prove \eqref{eq:induction2}
for $j=1$, write
\begin{align*}
  \|A(n\tau)Z^\phi_\tau(n\tau)\|_{\ell^\infty(I_1;L^2)} &\le
\|A(n\tau)u^\phi(n\tau)\|_{\ell^\infty(I_1;L^2)}\\
&\quad +\|A(n\tau)\(Z^\phi_\tau(n\tau)-u^\phi(n\tau)\)\|_{\ell^\infty(I_1;L^2)}\\
 &\le \frac{M}{2} + 3C_0\frac{M}{(10 C_0)^{ K}}\le M.
\end{align*}
Suppose now that \eqref{eq:induction1} and \eqref{eq:induction2} hold
for some $1\le j\le K-1$.
We invoke the same intermediary results as
in the case $j=1$:
\begin{align*}
  &\quad\|A\(Z^\phi_\tau(n\tau)-u^\phi(n\tau)\)\|_{\ell^\infty(I_{j+1};L^2)}\\
  &\le
\|A\(Z^\phi_\tau(n\tau)-Z^{\psi_{j+1}}_\tau(n\tau-m_{j+1}\tau)\)\|_{\ell^\infty(I_{j+1};L^2)}\\
 &\quad+
\|A\(Z^{\psi_{j+1}}_\tau(n\tau-m_{j+1}\tau)-u^{\psi_{j+1}}(n\tau-m_{j+1}\tau)\)\|_{\ell^\infty(I_{j+1};L^2)}
  \\
  &\quad+
  \|A\(u^{\psi_{j+1}}(n\tau-m_{j+1}\tau)-u^\phi(n\tau)\)\|_{\ell^\infty(I_{j+1};L^2)}.
\end{align*}
For the first term on the right hand side, in view of
\eqref{eq:induction2} at step $j$, we have:
\begin{align*}
  &\quad\|A(Z^\phi_\tau(n\tau)-Z^{\psi_{j+1}}_\tau(n\tau-
  m_{j+1}\tau))\|_{\ell^\infty(I_{j+1};L^2)}\\
                &\le C_0
 \|Z^\phi_\tau(m_{j+1}\tau)-\psi_{j+1}\|_{X(\{m_{j+1}\tau\})} \\
  &\le C_0\|Z^\phi_\tau(m_{j+1}\tau)-u^\phi(m_{j+1}\tau)\|_{X(\{m_{j+1}\tau\})}
  + C_0\|u^\phi(m_{j+1}\tau)-\psi_{j+1}\|_{X(\{m_{j+1}\tau\})}\\
  &\le C_0\max_{B\in
    \{\Id,\nabla,J\}}\|B(Z^\phi_\tau(n\tau)-u^\phi(n\tau))\|_{\ell^\infty(I_j;L^2)}
  + C_0 \frac{M}{(10 C_0)^{ K}}\\
  &\le C_0 \times (3C_0)^j  \frac{M}{(10 C_0)^{K}}+ C_0
    \frac{M}{(10 C_0)^{ K}},
\end{align*}
where we have used \eqref{eq:induction1} at step $j$, and
\eqref{eq:psij}. The second term is estimated in
view of \eqref{eq:local-error-H2}, by
\begin{equation*}
  \|A\(Z^{\psi_{j+1}}_\tau(n\tau-m_{j+1}\tau)-u^{\psi_{j+1}}(n\tau-m_{j+1}\tau)\)\|_{\ell^\infty(I_{j+1};L^2)} \le
  \tau^{1/2}C(d,\si,M,\psi_{j+1}).
\end{equation*}
For $0<\tau\le \tau_{j+1}=\tau_{j+1} (d,\si,M,\psi_{j+1})$, we infer
\begin{equation*}
 \|A\(Z^{\psi_{j+1}}_\tau(n\tau-m_{j+1}\tau)-u^{\psi_{j+1}}(n\tau-m_{j+1}\tau)\)\|_{\ell^\infty(I_{j+1};L^2)}  \le
  C_0\frac{M}{(10 C_0)^{K}} .
\end{equation*}
Eventually, we assume
\begin{equation*}
  0<\tau\le \min_{1\le j\le K}\tau_j.
\end{equation*}
Finally, Theorem~\ref{theo:B} and \eqref{eq:psij} yield
\begin{align*}
\|A\(u^{\psi_{j+1}}(n\tau-m_{j+1}\tau)-u^\phi(n\tau)\)\|_{\ell^\infty(I_{j+1};L^2)} &  \le
 C_0\|\psi_{j+1}-u^\phi(m_{j+1}\tau)\|_{X(\{m_{j+1}\tau\})}\\
&\le
   C_0\frac{M}{(10 C_0)^{K}} ,
\end{align*}
hence
\begin{align*}
  \|A\(Z^\phi_\tau(n\tau)-u^\phi(n\tau)\)\|_{\ell^\infty(I_{j+1};L^2)}&\le
  \(3^jC_0^{j+1} + 3C_0\)   \frac{M}{(10 C_0)^{K}}\\
  &\le \(3^jC_0^{j+1}
 + 3C_0^{j+1}\)\frac{M}{(10 C_0)^{K}},
\end{align*}
and \eqref{eq:induction1} follows for $j+1$. To prove
\eqref{eq:induction2}, write again
\begin{align*}
  \|A(n\tau)Z^\phi_\tau(n\tau)\|_{\ell^\infty(I_{j+1};L^2)} &\le
\|A(n\tau)u^\phi(n\tau)\|_{\ell^\infty(I_{j+1};L^2)}\\
&\quad +\|A(n\tau)\(Z^\phi_\tau(n\tau)-u^\phi(n\tau)\)\|_{\ell^\infty(I_{j+1};L^2)}\\
 &\le \frac{M}{2} + (3C_0)^{j+1}\frac{M}{(10 C_0)^{ K}}\le M.
\end{align*}
This yields Theorem~\ref{theo:stability} for $(q,r)=(\infty,2)$. The
case of other admissible pairs then follows from
Proposition~\ref{prop:local-stab}.
\end{proof}

\section{Conclusion and outlooks}
\label{sec:conclusion}

In this paper, we have proved the first uniform in time error estimate
for splitting methods in the context of nonlinear Schr\"odinger
equations. The main tools are discrete Strichartz estimates for the
frequency truncated free Schr\"odinger group established in
\cite{IgnatZuazua2009} on the one hand, the Galilean operator $J$
and its main properties discovered in \cite{GV79Scatt} on the other
hand. The Galilean operator provides precious dispersive estimates, in
the sense that $L^r$-norms in space decay with an explicit rate in
time: this is more precise than belonging to some space $L^q_tL^r_x$,
a property related to Strichartz estimates. We have considered
Lie-Trotter splitting in time, the question of higher order methods,
starting with Strang splitting, remains open.
\smallbreak

For several technical reasons (see Remark~\ref{rem:restriction}), we
have assumed $\si\ge 2/d$. Filling the gap $\si_*<\si<2/d$ does not
seem straightforward, and is a fairly natural question.
\smallbreak

It is very likely that the method presented here can be adapted in
order to treat some (defocusing) Hartree-type nonlinearities, of the form
\begin{equation*}
  \(\frac{1}{|x|^{\gamma}}\ast|u|^2\)u,
\end{equation*}
with $4/3<\gamma<\min (4,d)$; see \cite{GV80,HT87R3} for scattering theory
in this case.
\smallbreak

As it is the case in several physical models, such as Bose-Einstein
condensation (see e.g. \cite{JosserandPomeau}), one may ask if our
result remains true if \eqref{eq:NLS} is replaced by
\begin{equation*}
   i \partial_t u +\frac{1}{2} \Delta u =  |u|^{2\si}
   u+\frac{|x|^2}{2}u.
\end{equation*}
As the harmonic potential prevents large time dispersion (linear
solutions are periodic in time), no uniform in time error estimate
should be expected to the splitting method. More generally, if the
solution to the Schr\"odinger equation that we consider is not
(sufficiently) dispersive, then our approach fails, as discussed in
Remark~\ref{rem:non-disp}.
On the other hand, if one
considers the \emph{repulsive} harmonic potential,
\begin{equation*}
   i \partial_t u +\frac{1}{2} \Delta u =  |u|^{2\si}
   u-\frac{|x|^2}{2}u,
 \end{equation*}
 which enhances the linear dispersion (algebraic decay in time becomes
 an exponential decay),
 then the scattering theory developed in \cite{CaSIMA} suggests that an
 analogue of Theorem~\ref{theo:CVL2} may be available.
\smallbreak

In the case of linear Schr\"odinger equations with a potential,
\begin{equation*}
  i\d_t u+\frac{1}{2}\Delta u= Vu,
\end{equation*}
if  $V$ is a short-range potential in the sense that
scattering theory is available (see e.g. \cite{DG} for a reference
including time-dependent potentials), it seems that no uniform in time
error estimate is known for splitting methods. The case of short-range
potentials should be a reasonable framework to try to prove the
analogue of Theorem~\ref{theo:CVL2}, since solutions are
asymptotically free (the action of the potential $V$ is negligible for
large time), but a suitable technical approach is to be developed.
\smallbreak

Probably the most natural and challenging question at this stage would
be to adapt our result to the fully discrete case, that is, taking
spatial discretization into account. The first remark is that fully
discrete Strichartz estimates have been proven in
\cite{IgnatZuazua2009}. The most important aspect to address is
therefore the adaptation of the Galilean operator
$J(t)=x+it\nabla$. The two operators involved in the definition of $J$,
the multiplication by $x$ and the gradient in space, are easily
adapted to the discrete case, but the analogue of
Proposition~\ref{prop:pseudoconf} is certainly the key step to
understand, in order to treat the fully discrete case.

%%%%%%%%%%%%%%%%%%%%%%%%%%%%%%%%%%%%%%%%%%%%%%%%%%%%%%%%%%%%%%%%%%%%%%%%%%%%%%%%%%%%%%%%%%%%

\bibliographystyle{abbrv}
\bibliography{biblio}

\begin{thebibliography}{10}

\bibitem{BBD}
C.~Besse, B.~Bid{\'e}garay, and S.~Descombes.
\newblock Order estimates in time of splitting methods for the nonlinear
  {S}chr\"odinger equation.
\newblock {\em SIAM J. Numer. Anal.}, 40(1):26--40, 2002.

\bibitem{CaSIMA}
R.~Carles.
\newblock Nonlinear {S}chr\"odinger equations with repulsive harmonic potential
  and applications.
\newblock {\em SIAM J. Math. Anal.}, 35(4):823--843, 2003.

\bibitem{CaDPDE}
R.~Carles.
\newblock Rotating points for the conformal {NLS} scattering operator.
\newblock {\em Dyn. Partial Differ. Equ.}, 6(1):35--51, 2009.

\bibitem{CazCourant}
T.~Cazenave.
\newblock {\em Semilinear {S}chr\"odinger equations}, volume~10 of {\em Courant
  Lecture Notes in Mathematics}.
\newblock New York University Courant Institute of Mathematical Sciences, New
  York, 2003.

\bibitem{CW92}
T.~Cazenave and F.~Weissler.
\newblock Rapidly decaying solutions of the nonlinear {S}chr\"odinger equation.
\newblock {\em Comm. Math. Phys.}, 147:75--100, 1992.

\bibitem{ChoiKoh2021}
W.~{Choi} and Y.~{Koh}.
\newblock {On the splitting method for the nonlinear Schr\"odinger equation
  with initial data in $H^1$}.
\newblock {\em {Discrete Contin. Dyn. Syst.}}, 41(8):3837--3867, 2021.

\bibitem{DG}
J.~Derezi\'nski and C.~G\'erard.
\newblock {\em Scattering theory of quantum and classical {N}-particle
  systems}.
\newblock Texts and Monographs in Physics, Springer Verlag, Berlin Heidelberg,
  1997.

\bibitem{FaouGrebert2011}
E.~{Faou} and B.~{Gr\'ebert}.
\newblock {Hamiltonian interpolation of splitting approximations for nonlinear
  PDEs}.
\newblock {\em {Found. Comput. Math.}}, 11(4):381--415, 2011.

\bibitem{FHR18}
E.~{Faou}, R.~{Horsin}, and F.~{Rousset}.
\newblock {On numerical Landau damping for splitting methods applied to the
  Vlasov-HMF model}.
\newblock {\em {Found. Comput. Math.}}, 18(1):97--134, 2018.

\bibitem{Ginibre}
J.~Ginibre.
\newblock An introduction to nonlinear {S}chr\"odinger equations.
\newblock In R.~Agemi, Y.~Giga, and T.~Ozawa, editors, {\em Nonlinear waves
  (Sapporo, 1995)}, GAKUTO International Series, Math. Sciences and Appl.,
  pages 85--133. Gakk\={o}tosho, Tokyo, 1997.

\bibitem{GV79Cauchy}
J.~Ginibre and G.~Velo.
\newblock On a class of nonlinear {S}chr\"odinger equations. {I} {T}he {C}auchy
  problem, general case.
\newblock {\em J. Funct. Anal.}, 32:1--32, 1979.

\bibitem{GV79Scatt}
J.~Ginibre and G.~Velo.
\newblock On a class of nonlinear {S}chr\"odinger equations. {II} {S}cattering
  theory, general case.
\newblock {\em J. Funct. Anal.}, 32:33--71, 1979.

\bibitem{GV80}
J.~Ginibre and G.~Velo.
\newblock On a class of nonlinear {S}chr\"{o}dinger equations with nonlocal
  interaction.
\newblock {\em Math. Z.}, 170(2):109--136, 1980.

\bibitem{HT87}
N.~Hayashi and Y.~Tsutsumi.
\newblock Remarks on the scattering problem for nonlinear {S}chr\"odinger
  equations.
\newblock In {\em Differential equations and mathematical physics (Birmingham,
  Ala., 1986)}, volume 1285 of {\em Lectures Notes in Math.}, pages 162--168.
  Springer, Berlin, 1987.

\bibitem{HT87R3}
N.~Hayashi and Y.~Tsutsumi.
\newblock Scattering theory for {H}artree type equations.
\newblock {\em Ann. Inst. H. Poincar\'{e} Phys. Th\'{e}or.}, 46(2):187--213,
  1987.

\bibitem{Ignat2011}
L.~I. {Ignat}.
\newblock {A splitting method for the nonlinear Schr\"odinger equation}.
\newblock {\em {J. Differ. Equations}}, 250(7):3022--3046, 2011.

\bibitem{IgnatZuazua2006}
L.~I. {Ignat} and E.~{Zuazua}.
\newblock {Dispersive properties of numerical schemes for nonlinear
  Schr\"odinger equations}.
\newblock In {\em Foundations of computational mathematics, Santander 2005.
  Selected papers based on the presentations at the international conference of
  the Foundations of Computational Mathematics (FoCM), Santander, Spain, June
  30 -- July 9, 2005.}, pages 181--207. Cambridge: Cambridge University Press,
  2006.

\bibitem{IgnatZuazua2009}
L.~I. {Ignat} and E.~{Zuazua}.
\newblock {Numerical dispersive schemes for the nonlinear Schr\"odinger
  equation}.
\newblock {\em {SIAM J. Numer. Anal.}}, 47(2):1366--1390, 2009.

\bibitem{JosserandPomeau}
C.~Josserand and Y.~Pomeau.
\newblock Nonlinear aspects of the theory of {B}ose-{E}instein condensates.
\newblock {\em Nonlinearity}, 14(5):R25--R62, 2001.

\bibitem{KOPV17}
R.~Killip, T.~Oh, O.~Pocovnicu, and M.~Vi\c{s}an.
\newblock Solitons and scattering for the cubic-quintic nonlinear
  {S}chr\"{o}dinger equation on {$\mathbb{R}^3$}.
\newblock {\em Arch. Ration. Mech. Anal.}, 225(1):469--548, 2017.

\bibitem{Lu08}
C.~Lubich.
\newblock On splitting methods for {S}chr\"odinger-{P}oisson and cubic
  nonlinear {S}chr\"odinger equations.
\newblock {\em Math. Comp.}, 77(264):2141--2153, 2008.

\bibitem{NakanishiOzawa}
K.~Nakanishi and T.~Ozawa.
\newblock Remarks on scattering for nonlinear {S}chr\"odinger equations.
\newblock {\em NoDEA Nonlinear Differential Equations Appl.}, 9(1):45--68,
  2002.

\bibitem{ORS21}
A.~{Ostermann}, F.~{Rousset}, and K.~{Schratz}.
\newblock {Error estimates of a Fourier integrator for the cubic Schr\"odinger
  equation at low regularity}.
\newblock {\em {Found. Comput. Math.}}, 21(3):725--765, 2021.

\bibitem{StKe05}
A.~{Stefanov} and P.~G. {Kevrekidis}.
\newblock {Asymptotic behaviour of small solutions for the discrete nonlinear
  Schr\"odinger and Klein-Gordon equations}.
\newblock {\em {Nonlinearity}}, 18(4):1841--1857, 2005.

\bibitem{Stein70}
E.~M. Stein.
\newblock {\em Singular integrals and differentiability properties of
  functions}.
\newblock Princeton Mathematical Series, No. 30. Princeton University Press,
  Princeton, N.J., 1970.

\bibitem{TaoDisp}
T.~Tao.
\newblock {\em Nonlinear dispersive equations}, volume 106 of {\em CBMS
  Regional Conference Series in Mathematics}.
\newblock Published for the Conference Board of the Mathematical Sciences,
  Washington, DC, 2006.
\newblock Local and global analysis.

\end{thebibliography}
\end{document}